\documentclass[pdflatex,sn-mathphys-num]{sn-jnl}

\usepackage{amsmath,amssymb,amsthm,mathtools,amsfonts}
\usepackage{graphicx}
\usepackage{multirow}
\usepackage{mathrsfs}
\usepackage[title]{appendix}
\usepackage{xcolor}
\usepackage{textcomp}
\usepackage{manyfoot}
\usepackage{booktabs}
\usepackage{algorithm}
\usepackage{lineno}
\usepackage{algorithmicx}
\usepackage{algpseudocode}
\usepackage{listings}
\usepackage{hyperref}
\usepackage{enumitem}
\usepackage[mathscr]{eucal}

\hypersetup{
  colorlinks=true,
  linkcolor=blue!55!black,
  citecolor=blue!55!black,
  urlcolor=blue!55!black
}

\newcommand{\R}{\mathbb{R}}          

\newtheorem{theorem}{Theorem}[section]
\newtheorem{proposition}[theorem]{Proposition}
\newtheorem{lemma}[theorem]{Lemma}

\theoremstyle{definition}
\newtheorem{definition}[theorem]{Definition}
\newtheorem{assumption}[theorem]{Assumption}
\newtheorem{remark}[theorem]{Remark}
\newtheorem{example}[theorem]{Example}

\numberwithin{equation}{section}

\begin{document}

\title{Reflected Optimal Stopping with a Max-Type Payoff: Measure-Valued Stopping Gains and Killed Resolvent Representation}

\author[1]{\fnm{Ye} \sur{Liang}}
\email{ye-liang@uiowa.edu}

\affil[1]{\orgdiv{College of Engineering},
\orgname{The University of Iowa},
\orgaddress{
\city{Iowa City}, 
\state{IA},
\postcode{52242},      
\country{United States}
}}

\abstract{
We study an infinite-horizon optimal stopping problem for a two-dimensional
normally reflected diffusion in the quadrant with payoff
\(G(x_1,x_2)=x_1\vee \alpha x_2\). The problem combines three features that
complicate the usual free-boundary analysis: reflection on the coordinate axes,
a genuinely two-dimensional stopping region, and a nonsmooth max-type reward.
We formulate the associated reflected obstacle problem, prove a verification
theorem under explicit It\^o--Krylov--Tanaka admissibility and
measure-superharmonicity assumptions, and derive a conditional epigraph
structure for the stopping set. The main technical point is that the
stopping-gain object \(\Gamma=c+rG-\mathcal LG\) is a signed measure rather
than a function. Its diagonal component is
   $\Gamma^\Delta(dx)
   =
   -\frac{n^\top a(x)n}{2\sqrt{1+\alpha^2}}\sigma_\Delta(dx)$,
 $n=(1,-\alpha)$,
which shows that pointwise stopping-gain sign conditions must be interpreted
with care. We also prove that the correct potential representation is the
killed-resolvent formula
   $V(x)=G(x)-R_r^{\mathcal C}\Gamma(x)$,
rather than the unrestricted reflected resolvent. A constant-coefficient
reflected Brownian example illustrates the diagonal singular term explicitly.}



\maketitle

\section{Introduction}
\label{sec:introduction}

\subsection{The problem}
\label{sec:intro-problem}

This paper studies the infinite-horizon optimal stopping problem
\begin{equation}
\label{eq:intro-value}
   V(x)
   =
   \sup_{\tau\in\mathcal T}
   \mathbb E_x\left[
      e^{-r\tau}G(X_\tau)
      -
      \int_0^\tau e^{-rs}c(X_s)\,ds
   \right],
   \qquad
   G(x_1,x_2)=x_1\vee\alpha x_2,
\end{equation}
where \(r>0\), \(\alpha>0\), \(c\ge0\), and \(X=(X^1,X^2)\) is a normally
reflected diffusion in the quadrant \(\R_+^2=[0,\infty)^2\). Thus \(X\) evolves
in the interior according to a second-order diffusion generator
\[
   \mathcal L f
   =
   \sum_{i=1}^2\mu_i\partial_i f
   +
   \frac12\sum_{i,j=1}^2a_{ij}\partial_{ij}f,
   \qquad a=\sigma\sigma^\top,
\]
and is pushed back into \(\R_+^2\) by normal reflection on the coordinate axes.
The problem is therefore simultaneously constrained, multidimensional, and
nonsmooth: one has reflection at the axes, a max-type payoff with a kink on
\(\Delta:=\{x_1=\alpha x_2\}\), and a genuinely two-dimensional stopping
region. These features force one to work with reflected boundary conditions,
measure-valued stopping gains, and killed rather than unrestricted potentials.

The max reward \(x_1\vee\alpha x_2\) is natural in optimal stopping problems
related to American options on the maximum of several assets, exchange-type
contracts, and dual-strike payoffs; see, for example,
\cite{broadie1997valuation,detemple2014optimal,wang2025analysis,detemple2005american,liu2026computational}. The
geometry of the associated exercise regions in the multi-asset setting was
analyzed by Villeneuve \cite{villeneuve1999exercise}, and the general
valuation and computation theory is surveyed in
\cite{karatzas1998methods,wang2025analysis1,detemple2005american,haugh2004pricing,khaliq2009new,company2023etd,poncet2022american,wang2026optimal,yu2026endogenous,yu2026fredholm}. In the present paper, the
financial interpretation is secondary. The main point is structural: the
combination of normal reflection and the nonsmooth max payoff creates a
singular stopping-gain measure and invalidates several formal free-boundary
steps that are harmless in smoother one-dimensional models.

\subsection{Related literature}
\label{sec:intro-literature}

The connection between optimal stopping, Snell envelopes, variational
inequalities, and free-boundary problems is classical; we use the monograph of
Peskir and Shiryaev \cite{peskir2006optimal} as the main reference for this
probabilistic and analytic framework, and refer to
\cite{shiryaev2008optimal,bensoussan2011applications,liu2025bidirectional,krylov1980controlled,%
oksendal2013stochastic,pham2009continuous,yu2026bilinear,wang2026vital,wang2026finite} for the broader stochastic-control
and free-boundary background. Since the value function in
\eqref{eq:intro-value} is not assumed to be smooth, the obstacle problem must be
understood in a weak sense. Our sign convention is the reflected version of
   $\max\{(\mathcal L-r)V-c,\;G-V\}=0$,
and the viscosity terminology follows the standard framework of
Crandall, Ishii, and Lions \cite{crandall1992user}; for the interplay between
viscosity solutions, obstacle problems, and reflected backward equations we
refer to \cite{fleming2006controlled,liang2025global,yu2026optimization3,reikvam1998viscosity,wang2026first,wang2026introduction,yu2026structural,el1997reflected}.
The optimal stopping of one-dimensional diffusions has a complete and elegant
theory through the concavity/excessivity characterization of Dayanik and
Karatzas \cite{dayanik2003optimal}, while genuinely multidimensional problems,
including the regularity and continuity of the stopping boundary, remain
comparatively delicate; see Peskir \cite{peskir2019continuity,yu2026optimization2} and Christensen,
Crocce, Mordecki, and Salminen \cite{christensen2019optimal,yu2026optimization1}.

The reflected nature of \(X\) is not a cosmetic boundary condition. Existence,
uniqueness, continuity in the initial state, and the form of the boundary local
time depend on the Skorokhod problem, the geometry of the domain, and the
reflection field. The foundational reflected-SDE theory goes back, in
particular, to the submartingale-problem approach of Stroock and Varadhan
\cite{stroock1971diffusion} and to the deterministic Skorokhod-map estimates of
Lions and Sznitman \cite{lions1984stochastic}, with Lipschitz continuity of the
reflection map established by Dupuis and Ishii \cite{dupuis1991lipschitz}.
Related orthant and nonsmooth-domain constructions appear in
\cite{tanaka1979stochastic,wang2026algebraic,harrison1981reflected,saisho1987stochastic,wang2026breakdown,taylor1993existence,burdzy2017obliquely,gao2022rolling,wang2026elliptic,hino2021pathwise,wei2026mckean,yu2026microscopic}. Reflected Brownian motion in an orthant, its stationary
behavior, and its boundary structure have been studied extensively; see
Harrison and Williams \cite{harrison1987multidimensional,wang2025well}, Williams
\cite{williams1987reflected}, Reiman and Williams \cite{reiman1988boundary,yu2026controlling},
Kang and Williams \cite{kang2007invariance,yu2026diagnostic}, Hobson and Rogers
\cite{hobson1993recurrence} for the quadrant case, and Lipshutz and Ramanan
\cite{lipshutz2019pathwise} for pathwise differentiability in polyhedral
domains. The
measure-theoretic language used below is motivated by the distributional
formulation of optimal stopping and potential theory: Lamberton and Zervos
\cite{lamberton2013optimal} characterize one-dimensional optimal stopping
problems through variational inequalities in the sense of distributions and
additive-functional potentials, while the Revuz correspondence and Dirichlet
form machinery provide the general language for smooth measures and continuous
additive functionals \cite{revuz2013continuous,yu2026rigorous,wang2026damage,fukushima2011dirichlet,wang2026lecture,yu2026mode,wang2026introduction2,jie2026optimal}. The
local-time and Itô--Tanaka apparatus underlying the singular part of the
stopping-gain measure is standard semimartingale calculus
\cite{revuz2013continuous,wang2025hybrid,yu2026beyond,protter2004stochastic,yu2026pattern,yu2026chemotactic,karatzas2014brownian,wang2025multi}, and the
general potential-theoretic framework for killed processes and additive
functionals may be found in Blumenthal and Getoor
\cite{blumenthal2007markov}.

\subsection{Three gaps addressed in this paper}
\label{sec:intro-gaps}

We focus on three technical points that are often hidden in formal treatments
of problems like \eqref{eq:intro-value}.

First, viscosity supersolution status is not a substitute for the Sobolev,
distributional, or additive-functional structure needed in an
It\^o--Krylov--Tanaka verification argument. A continuous function satisfying
a viscosity inequality need not, by that fact alone, possess a signed-measure
representative of \((\mathcal L-r)u-c\), nor does it automatically belong to a
class for which a generalized It\^o formula is available. We therefore state
the verification theorem under explicit measure-superharmonicity and
It\^o--Krylov--Tanaka admissibility assumptions.

Second, the stopping-gain object is a signed measure, not an ordinary function.
With the convention
   $\Gamma:=c+rG-\mathcal LG$,
the absolutely continuous part of \(\Gamma\) is obtained away from
\(\Delta=\{x_1=\alpha x_2\}\), but the kink of \(G\) creates an additional
surface measure on \(\Delta\). Writing \(Y(x):=x_1-\alpha x_2\),
\(n:=\nabla Y=(1,-\alpha)\), and \(q(x):=n^\top a(x)n\), the diagonal component
is
\[
   \Gamma^\Delta(dx)
   =
   -\frac{q(x)}{2\sqrt{1+\alpha^2}}\,\sigma_\Delta(dx),
\]
where \(\sigma_\Delta\) denotes one-dimensional surface measure on
\(\Delta\). Since \(q\ge0\), this singular component is nonpositive. Thus a
pointwise condition such as ``\(\Gamma\ge0\) on the stopping set'' is not a
literal statement about the full measure when the stopping set intersects
\(\Delta\) in positive length; the absolutely continuous and singular parts
must be separated.

Third, the potential representation must be killed at the stopping time. If
\(\mathcal D:=\{V=G\}\), \(\mathcal C:=\{V>G\}\), and
\(\tau_{\mathcal D}:=\inf\{t\ge0:X_t\in\mathcal D\}\), then the correct
pre-stopping potential is
\[
   R_r^{\mathcal C}\Gamma(x)
   :=
   \mathbb E_x\left[
      \int_0^{\tau_{\mathcal D}}e^{-rs}\,dA_s^\Gamma
   \right],
\]
where \(A^\Gamma\) is the signed continuous additive functional associated
with \(\Gamma\). The value representation is
   $V(x)=G(x)-R_r^{\mathcal C}\Gamma(x)$.
In general this is not the same as
\(G(x)-R_r^{\mathrm R}(\Gamma\mathbf 1_{\mathcal C})(x)\), because the
unrestricted reflected process continues after \(\tau_{\mathcal D}\) and may
re-enter \(\mathcal C\), accumulating occupation that is irrelevant to the
stopped problem. The optimal stopping problem sees only the path before the
first entry into \(\mathcal D\), so the resolvent must be killed at
\(\tau_{\mathcal D}\).

\subsection{Contributions}
\label{sec:intro-contributions}

The contributions of the paper are as follows.

First, we prove a verification theorem under explicit
measure-superharmonicity assumptions. The theorem separates the verification
argument from any hidden regularity theorem: the Sobolev, signed-measure,
boundary, and integrability hypotheses are named and used directly.

Second, we establish a conditional epigraph theorem for the stopping set. The
result is based on monotonicity of the stopping advantage \(H:=V-G\), not on an
unsupported order-preservation argument for \(V\) alone. Under vertical
monotonicity of \(H\) and nonempty vertical stopping sections, the stopping
region has the form \(\mathcal D=\{(x_1,x_2)\in\R_+^2:x_2\ge b(x_1)\}\);
additional horizontal monotonicity of \(H\) gives monotonicity of \(b\).

Third, we compute explicitly the singular stopping-gain measure generated by
the kink of \(G\) on \(\Delta\). The computation identifies the diagonal
surface measure
\[
   \Gamma^\Delta(dx)
   =
   -\frac{n^\top a(x)n}{2\sqrt{1+\alpha^2}}\,\sigma_\Delta(dx),
\]
and explains why the usual stopping-gain sign condition must be interpreted
with care in the presence of a max-type payoff.

Fourth, we derive the killed-resolvent representation
\(V=G-R_r^{\mathcal C}\Gamma\). This corrects the generally invalid
unrestricted-resolvent expression and makes explicit that the potential is
accumulated only before the first entry into the stopping set.

Fifth, we give a candidate-boundary verification theorem. A proposed boundary
is not validated merely by solving a boundary-trace or Fredholm-type equation;
one must verify majorization, contact, continuation-side dynamics, reflected
boundary behavior, global measure-superharmonicity, and the required
integrability assumptions. The main novelty is the systematic correction of
three common formal steps: regularity is not inferred from viscosity status,
the kink contribution is treated as a signed surface measure, and the
resolvent is killed at the optimal stopping time.

\subsection{Organization}
\label{sec:intro-organization}

Section~\ref{sec:model} introduces the reflected diffusion, the optimal
stopping problem, the obstacle convention, and the stopping-gain measure.
Section~\ref{sec:main-results} states the main results. Section~\ref{sec:proofs}
gives the proofs. Section~\ref{sec:conclusion} concludes. Technical lemmas
concerning dynamic programming, Lyapunov estimates, corner reflection, and
signed additive functionals are collected in the Appendix.
\section{The Model}
\label{sec:model}

This section fixes the probabilistic setting, the optimal stopping problem, the
obstacle convention, and the signed stopping-gain measure used throughout the
paper. All regularity, integrability, and additive-functional assumptions are
stated explicitly; technical sufficient conditions are deferred to the
Appendix.

\subsection{Reflected diffusion in the quadrant}
\label{sec:model-reflected-diffusion}

Let \(\mathbb R_+^2:=[0,\infty)^2\) and
\(\mathbb R_{++}^2:=(0,\infty)^2\). On a filtered probability space
\((\Omega,\mathcal F,(\mathcal F_t)_{t\ge0},\mathbb P)\) satisfying the usual
conditions, let \(W=(W^1,W^2)\) be a two-dimensional Brownian motion. For each
initial state \(x\in\mathbb R_+^2\), we consider the normally reflected
diffusion \(X=X^x\) solving
\begin{equation}
\label{eq:reflected-sde}
   X_t
   =
   x+\int_0^t\mu(X_s)\,ds
     +\int_0^t\sigma(X_s)\,dW_s
     +L_t,
   \qquad t\ge0,
\end{equation}
where \(L=(L^1,L^2)\) is the boundary reflection process. Normal reflection
means that, for \(i=1,2\), \(L^i\) is continuous, nondecreasing, \(L_0^i=0\),
and
\begin{equation}
\label{eq:normal-reflection}
   \int_0^\infty \mathbf 1_{\{X_t^i>0\}}\,dL_t^i=0.
\end{equation}
Thus \(dL^i\) is carried by the face \(\{x_i=0\}\). We write
\(a(x):=\sigma(x)\sigma(x)^\top\). For \(f\in C^2(\mathbb R_{++}^2)\), the
interior generator is
\[
   \mathcal L f(x)
   =
   \sum_{i=1}^2\mu_i(x)\partial_i f(x)
   +
   \frac12\sum_{i,j=1}^2a_{ij}(x)\partial_{ij}f(x),
   \qquad x\in\mathbb R_{++}^2.
\]

\begin{assumption}[Standing assumptions on the reflected diffusion]
\label{ass:reflected-diffusion}
Throughout the paper the following hypotheses are imposed.
\begin{enumerate}
\item[\textup{(A1)}]
The coefficients \(\mu:\mathbb R_+^2\to\mathbb R^2\) and
\(\sigma:\mathbb R_+^2\to\mathbb R^{2\times2}\) are locally Lipschitz and have
at most linear growth.

\item[\textup{(A2)}]
The matrix \(a=\sigma\sigma^\top\) is locally uniformly elliptic in
\(\mathbb R_{++}^2\): for every compact \(K\Subset\mathbb R_{++}^2\) there is
\(\lambda_K>0\) such that \(\xi^\top a(x)\xi\ge\lambda_K|\xi|^2\) for all
\(x\in K\) and \(\xi\in\mathbb R^2\).

\item[\textup{(A3)}]
For every \(x\in\mathbb R_+^2\), the reflected SDE
\eqref{eq:reflected-sde}--\eqref{eq:normal-reflection} admits a unique strong
solution, and the family \((X^x)_{x\in\mathbb R_+^2}\) is strong Markov.

\item[\textup{(A4)}]
The solution is continuous in the initial state: for every \(T>0\), every
compact \(K\subset\mathbb R_+^2\), and some \(q\ge1\),
\[
   \lim_{y\to x}
   \mathbb E\!\left[\sup_{0\le t\le T}|X_t^y-X_t^x|^q\right]=0,
   \qquad x\in K.
\]

\item[\textup{(A5)}]
The discount rate satisfies \(r>0\), the payoff parameter satisfies
\(\alpha>0\), and the running cost \(c:\mathbb R_+^2\to[0,\infty)\) is
continuous. Additional local regularity of \(c\) is imposed only where it is
needed.

\item[\textup{(A6)}]
The value function \(V\) defined in \eqref{eq:value-function} below is finite
and has polynomial growth: there exist constants \(C_V>0\) and \(p\ge1\) such
that \(0\le V(x)\le C_V(1+|x|^p)\) for all \(x\in\mathbb R_+^2\). The
discounted payoff and running-cost terms appearing below are assumed
integrable for the stopping times under consideration.
\end{enumerate}
\end{assumption}

\begin{remark}[Status of the reflected-SDE assumptions]
\label{rem:reflected-sde-background}
Assumption~\ref{ass:reflected-diffusion} is part of the model, not a theorem
proved in this paper. For reflected SDEs, existence, uniqueness, Markov
properties, and stability depend on the Skorokhod problem, the domain, the
reflection field, and the coefficient regularity. In the present quadrant with
normal reflection, these issues are simpler than in general nonsmooth domains
or for oblique reflection, but they are still not consequences of the displayed
SDE alone.
\end{remark}

\subsection{The optimal stopping problem}
\label{sec:model-optimal-stopping}

Let \(\mathcal T\) denote the set of all \((\mathcal F_t)\)-stopping times with
values in \([0,\infty]\). We use the convention
\(e^{-r\tau}G(X_\tau):=0\) on \(\{\tau=\infty\}\). The reward is
   $G(x_1,x_2):=x_1\vee\alpha x_2$,
 $\alpha>0$,
and the value function is
\begin{equation}
\label{eq:value-function}
   V(x)
   =
   \sup_{\tau\in\mathcal T}
   \mathbb E_x\left[
      e^{-r\tau}G(X_\tau)
      -
      \int_0^\tau e^{-rs}c(X_s)\,ds
   \right],
   \qquad x\in\mathbb R_+^2 .
\end{equation}
The stopping region, continuation region, and stopping advantage are
\[
   \mathcal D:=\{x\in\mathbb R_+^2:V(x)=G(x)\},
   \qquad
   \mathcal C:=\{x\in\mathbb R_+^2:V(x)>G(x)\},
   \qquad
   H:=V-G .
\]
Immediate stopping gives
   $V(x)\ge G(x)$, $x\in\mathbb R_+^2$,
so \(H\ge0\), \(\mathcal D=\{H=0\}\), and \(\mathcal C=\{H>0\}\). Whenever
\(V\) is continuous, \(\mathcal D\) is closed and \(\mathcal C\) is open.

\subsection{Obstacle problem and reflected viscosity convention}
\label{sec:model-obstacle}

Formally, the value function satisfies the obstacle problem
\begin{equation}
\label{eq:obstacle-max}
   \max\{(\mathcal L-r)V-c,\;G-V\}=0
   \quad\text{in }\mathbb R_{++}^2,
\end{equation}
together with the normal-reflection condition
   $\partial_iV=0$ on $\{x_i=0\}$, $i=1,2$.
The equation is understood in the reflected viscosity sense: test functions
touching from above satisfy the subsolution inequality, test functions touching
from below satisfy the supersolution inequality, and the Neumann condition is
imposed through the standard relaxed boundary condition. To avoid ambiguity at
the corner, define the active boundary index set
   $I(x):=\{i\in\{1,2\}:x_i=0\}$.
Thus \(I(x)=\varnothing\) for \(x\in\mathbb R_{++}^2\), \(I(x)=\{i\}\) on the
relative interior of the face \(\{x_i=0\}\), and \(I(0,0)=\{1,2\}\). With
\(F[\varphi](x):=\max\{(\mathcal L-r)\varphi(x)-c(x),\,G(x)-\varphi(x)\}\),
the relaxed reflected-viscosity boundary convention is encoded by
\[
   \min\!\left\{F[\varphi](x),\,\min_{i\in I(x)}\partial_i\varphi(x)\right\}
   \le0
   \quad\text{for upper tests},
\]
and
\[
   \max\!\left\{F[\varphi](x),\,\max_{i\in I(x)}\partial_i\varphi(x)\right\}
   \ge0
   \quad\text{for lower tests},
\]
with the convention that the boundary term is omitted when \(I(x)=\varnothing\).
The full formal definition and the dynamic-programming justification are
collected in Appendix~\ref{sec:appendix-dpp-viscosity}.

\begin{remark}[Sign convention]
\label{rem:obstacle-sign-convention}
The max convention \eqref{eq:obstacle-max} means that, in the continuation
region \(\mathcal C\), one expects \((\mathcal L-r)V-c=0\), while on the
stopping region \(\mathcal D\), where \(V=G\), the obstacle inequality gives
\((\mathcal L-r)G-c\le0\) in the appropriate weak sense. Equivalently, the
stopping-gain convention \(\Gamma:=c+rG-\mathcal LG\) is chosen so that the
formal stopping-region inequality becomes \(\Gamma\ge0\) away from singular
sets.
\end{remark}

\begin{remark}[Displayed sign check for the obstacle convention]
\label{rem:displayed-obstacle-sign-check}
The max-form obstacle convention used in \eqref{eq:obstacle-max} is
   $\max\{(\mathcal L-r)V-c,\;G-V\}=0$.
Thus, in the continuation set \(\mathcal C=\{V>G\}\), one has \(G-V<0\).
Since the maximum of the two terms is zero, the first term must be zero:
\[
   x\in\mathcal C
   \quad\Longrightarrow\quad
   G(x)-V(x)<0,
   \qquad
   \max\{(\mathcal L-r)V(x)-c(x),\,G(x)-V(x)\}=0,
\]
and therefore
   $(\mathcal L-r)V(x)-c(x)=0$,
 $x\in\mathcal C\cap\mathbb R_{++}^2$.
Equivalently, in the continuation region,
   $\mathcal LV(x)-rV(x)=c(x)$.
On the stopping set \(\mathcal D=\{V=G\}\), the obstacle term vanishes:
\(G-V=0\). Substituting \(V=G\) formally into the obstacle inequality gives
\[
   x\in\mathcal D
   \quad\Longrightarrow\quad
   \max\{(\mathcal L-r)G(x)-c(x),\,0\}=0,
\]
and hence
$(\mathcal L-r)G(x)-c(x)\le0$,
$x\in\mathcal D\cap(\mathbb R_{++}^2\setminus\Delta)$,
where the restriction away from \(\Delta=\{x_1=\alpha x_2\}\) is needed because
\(G\) is only piecewise \(C^2\).
This explains the sign convention for the stopping-gain measure. Away from
\(\Delta\),
   $\Gamma
   :=
   c+rG-\mathcal LG
   =
   -\big[(\mathcal L-r)G-c\big]$.
Therefore, on smooth parts of the stopping set,
\[
   (\mathcal L-r)G-c\le0
   \quad\Longleftrightarrow\quad
   -\big[(\mathcal L-r)G-c\big]\ge0
   \quad\Longleftrightarrow\quad
   \Gamma\ge0.
\]
Equivalently, in the two smooth regions
   $\mathcal R_1:=\{x_1>\alpha x_2\}$,
   $\mathcal R_2:=\{x_1<\alpha x_2\}$,
one has
\[
   G=x_1,\quad \mathcal LG=\mu_1,\quad
   \Gamma=c+rx_1-\mu_1
   \qquad \text{on }\mathcal R_1,
\]
and
\[
   G=\alpha x_2,\quad \mathcal LG=\alpha\mu_2,\quad
   \Gamma=c+r\alpha x_2-\alpha\mu_2
   \qquad \text{on }\mathcal R_2.
\]
Thus the formal stopping-region sign condition is
   $c(x)+rx_1-\mu_1(x)\ge0$ on $\mathcal D\cap\mathcal R_1$,
and
   $c(x)+r\alpha x_2-\alpha\mu_2(x)\ge0$
   on $\mathcal D\cap\mathcal R_2$.
At the kink \(\Delta\), this pointwise computation is no longer valid because
\(G\notin C^2\). Instead,
\[
   \Gamma=\Gamma^{\rm ac}\,dx+\Gamma^\Delta,
   \qquad
   \Gamma^\Delta(dx)
   =
   -\frac{q(x)}{2\sqrt{1+\alpha^2}}\,\sigma_\Delta(dx),
   \qquad
   q(x)=n^\top a(x)n,\quad n=(1,-\alpha).
\]
Hence the notation ``\(\Gamma\ge0\) on \(\mathcal D\)'' is meaningful only for
the absolutely continuous part away from \(\Delta\), or after explicitly
separating the diagonal singular component.
\end{remark}

\begin{remark}[Reflected It\^o formula and the Neumann boundary term]
\label{rem:reflected-ito-boundary-term}
Let \(f\in C^2(\mathbb R_+^2)\). For the normally reflected diffusion
\[
   dX_t=\mu(X_t)\,dt+\sigma(X_t)\,dW_t+dL_t,
   \qquad
   dL_t=(dL_t^1,dL_t^2),
\]
It\^o's formula gives
\[
\begin{aligned}
   df(X_t)
   &=
   \sum_{i=1}^2\partial_i f(X_t)\,dX_t^i
   +\frac12\sum_{i,j=1}^2\partial_{ij}f(X_t)\,d\langle X^i,X^j\rangle_t  \\
   &=
   \mathcal L f(X_t)\,dt
   +\nabla f(X_t)\sigma(X_t)\,dW_t
   +\partial_1f(X_t)\,dL_t^1
   +\partial_2f(X_t)\,dL_t^2 ,
\end{aligned}
\]
where
\[
   \mathcal L f(x)
   =
   \sum_{i=1}^2\mu_i(x)\partial_i f(x)
   +
   \frac12\sum_{i,j=1}^2a_{ij}(x)\partial_{ij}f(x),
   \qquad
   a=\sigma\sigma^\top .
\]
Equivalently, after discounting,
\begin{equation}
\label{eq:reflected-ito-discounted}
\begin{aligned}
   d\big(e^{-rt}f(X_t)\big)
   &=
   e^{-rt}\big[(\mathcal L-r)f\big](X_t)\,dt
   +e^{-rt}\nabla f(X_t)\sigma(X_t)\,dW_t  \\
   &\quad
   +e^{-rt}\partial_1f(X_t)\,dL_t^1
   +e^{-rt}\partial_2f(X_t)\,dL_t^2 .
\end{aligned}
\end{equation}
Integrating \eqref{eq:reflected-ito-discounted} over \([0,t]\) yields
\[
\begin{aligned}
   e^{-rt}f(X_t)
   &=
   f(x)
   +\int_0^t e^{-rs}\big[(\mathcal L-r)f\big](X_s)\,ds
   +\int_0^t e^{-rs}\nabla f(X_s)\sigma(X_s)\,dW_s  \\
   &\quad
   +\sum_{i=1}^2
      \int_0^t e^{-rs}\partial_i f(X_s)\,dL_s^i .
\end{aligned}
\]
The last term is the reflection contribution. Since \(L^i\) increases only on
the \(i\)-th face, i.e.
   $\int_0^\infty \mathbf 1_{\{X_s^i>0\}}\,dL_s^i=0$, $i=1,2$,
we have, for each \(t\ge0\),
\[
\begin{aligned}
   \int_0^t e^{-rs}\partial_i f(X_s)\,dL_s^i
   &=
   \int_0^t e^{-rs}\partial_i f(X_s)
      \mathbf 1_{\{X_s^i=0\}}\,dL_s^i  +
   \int_0^t e^{-rs}\partial_i f(X_s)
      \mathbf 1_{\{X_s^i>0\}}\,dL_s^i       \\
   &=
   \int_0^t e^{-rs}\partial_i f(X_s)
      \mathbf 1_{\{X_s^i=0\}}\,dL_s^i .
\end{aligned}
\]
Therefore, if \(f\) satisfies the normal-reflection condition
   $\partial_i f(x)=0$ for $x_i=0$, $i=1,2$,
then
\[
   \int_0^t e^{-rs}\partial_i f(X_s)\,dL_s^i=0,
   \qquad i=1,2,
\]
and the reflected It\^o formula reduces to
\[
   e^{-rt}f(X_t)
   =
   f(x)
   +\int_0^t e^{-rs}\big[(\mathcal L-r)f\big](X_s)\,ds
   +\int_0^t e^{-rs}\nabla f(X_s)\sigma(X_s)\,dW_s .
\]
This computation explains why the Neumann trace \(\partial_i f=0\) on
\(\{x_i=0\}\) is the correct analytic counterpart of normal reflection. It is
also the calculation used in the Lyapunov estimate: if
\(\partial_i\Psi=0\) on \(\{x_i=0\}\), then the reflection term in the formula
for \(e^{-rt}\Psi(X_t)\) vanishes, leaving only the drift
\((\mathcal L-r)\Psi\) and the martingale part.
\end{remark}

\subsection{The stopping-gain measure}
\label{sec:model-stopping-gain}

The key object used in the potential representation is the stopping-gain
measure
   $\Gamma:=c+rG-\mathcal LG$.
Since \(G(x_1,x_2)=x_1\vee\alpha x_2\) is not \(C^2\) across the kink set
   $\Delta:=\{(x_1,x_2)\in\mathbb R_+^2:x_1=\alpha x_2\}$,
\(\Gamma\) is not an ordinary function. It is a signed Radon measure with an
absolutely continuous part away from \(\Delta\) and a singular surface part on
\(\Delta\).

Define
\[
   Y(x):=x_1-\alpha x_2,
   \qquad
   n:=\nabla Y=(1,-\alpha),
   \qquad
   q(x):=n^\top a(x)n
        =a_{11}(x)-2\alpha a_{12}(x)+\alpha^2a_{22}(x).
\]
Since \(a(x)\) is nonnegative definite, \(q(x)\ge0\). On the two smooth
regions
   $\mathcal R_1:=\{x\in\mathbb R_+^2:x_1>\alpha x_2\}$,
   $\mathcal R_2:=\{x\in\mathbb R_+^2:x_1<\alpha x_2\}$,
the payoff is affine, namely \(G=x_1\) on \(\mathcal R_1\) and
\(G=\alpha x_2\) on \(\mathcal R_2\). Hence the absolutely continuous part of
\(\Gamma\) is
\[
   \Gamma^{\rm ac}(x)
   =
   \begin{cases}
   c(x)+rx_1-\mu_1(x), & x\in\mathcal R_1,\\[2mm]
   c(x)+r\alpha x_2-\alpha\mu_2(x), & x\in\mathcal R_2.
   \end{cases}
\]
The singular diagonal component is
\begin{equation}
\label{eq:Gamma-delta}
   \Gamma^\Delta(dx)
   =
   -\dfrac{q(x)}{2\sqrt{1+\alpha^2}}\,\sigma_\Delta(dx),
\end{equation}
where \(\sigma_\Delta\) denotes one-dimensional surface measure on
\(\Delta\). Thus the full decomposition is
\begin{equation}
\label{eq:Gamma-decomposition}
   \Gamma=\Gamma^{\rm ac}\,dx+\Gamma^\Delta .
\end{equation}
The proof of \eqref{eq:Gamma-delta} is given in
Section~\ref{sec:proof-singular-measure} using Tanaka's formula for
\(Y(X)=X^1-\alpha X^2\), the identity
\(G=(x_1+\alpha x_2+|Y|)/2\), and the co-area conversion
\(\delta_0(Y(x))\,dx=(1/\sqrt{1+\alpha^2})\,\sigma_\Delta(dx)\).

\begin{remark}[Why \(\Gamma\) must be treated as a measure]
\label{rem:Gamma-measure-warning}
The sign of the diagonal component is fixed by
\eqref{eq:Gamma-delta}: since \(q\ge0\), one has \(\Gamma^\Delta\le0\) as a
signed measure. Therefore the formal stopping-gain condition
\(\Gamma\ge0\) cannot be imposed literally on the full measure whenever the
stopping set contains a positive-length portion of \(\Delta\). In all later
statements, positivity or sign conditions involving \(\Gamma\) must specify
whether they concern the absolutely continuous part away from \(\Delta\), the
full signed measure, or a decomposition in which the diagonal component is
handled separately.
\end{remark}
\section{Main Results}
\label{sec:main-results}

This section states the main results of the paper. The proofs are postponed to
Section~\ref{sec:proofs}. The guiding point is that each formal step in the
optimal stopping calculation is stated with its necessary hypotheses: the
verification theorem uses measure-superharmonicity, the geometry theorem uses
monotonicity of \(H=V-G\), the kink of \(G\) is treated as a signed surface
measure, and the potential representation is killed at the first entry into
the stopping set.

\subsection{Verification under measure superharmonicity}
\label{sec:main-verification}

The first result is a verification theorem. Its purpose is not to prove
regularity of a viscosity solution, but to show that a sufficiently regular
majorant of the payoff is the value function once it satisfies the continuation
equation, the global measure-superharmonic inequality, the reflected boundary
condition, and the required integrability properties.

\begin{assumption}[It\^o--Krylov--Tanaka admissibility]
\label{ass:ito-krylov-admissibility}
Let \(u:\mathbb R_+^2\to\mathbb R\) be continuous and of polynomial growth.
We say that \(u\) is It\^o--Krylov--Tanaka admissible for the reflected
diffusion \(X\) if the following hold.
\begin{enumerate}
\item[\textup{(I1)}]
The distribution \(\mu_u:=(\mathcal L-r)u-c\), initially defined on
\(\mathbb R_{++}^2\), extends to a signed Radon measure on
\(\mathbb R_{++}^2\). Its Jordan decomposition
\(\mu_u=\mu_u^+-\mu_u^-\) consists of smooth measures for \(X\), and the
associated continuous additive functional is
\(A^{\mu_u}:=A^{\mu_u^+}-A^{\mu_u^-}\).

\item[\textup{(I2)}]
For every relatively open set \(U\subseteq\mathbb R_+^2\), every stopping time
\(\eta\), and every localizing sequence \((\eta_n)\), the stopped
It\^o--Krylov--Tanaka formula holds on
\([0,\eta\wedge\eta_n]\):
\[
   e^{-r(\eta\wedge\eta_n)}u(X_{\eta\wedge\eta_n})
   =
   u(x)
   +\int_0^{\eta\wedge\eta_n}e^{-rs}c(X_s)\,ds
   +\int_0^{\eta\wedge\eta_n}e^{-rs}\,dA_s^{\mu_u}
   +B_{\eta\wedge\eta_n}^u
   +M_{\eta\wedge\eta_n}^u,
\]
where \(M^u\) is a local martingale and \(B^u\) is the reflected boundary
contribution.

\item[\textup{(I3)}]
The reflected boundary contribution is nonpositive in the supermartingale
calculation: for every admissible localization,
\(\mathbb E_x[B_{\eta\wedge\eta_n}^u]\le0\). In particular, this holds if
\(u\) satisfies the normal-reflection condition \(\partial_i u=0\) on
\(\{x_i=0\}\), \(i=1,2\), in the relevant trace sense.

\item[\textup{(I4)}]
The stopped martingale terms have zero expectation after localization, and the
uniform-integrability conditions needed to pass \(n\to\infty\) are satisfied
for \(u(X)\), \(G(X)\), \(c(X)\), and \(A^{\mu_u}\).
\end{enumerate}
\end{assumption}

\begin{theorem}[Verification under measure superharmonicity]
\label{thm:verification-measure}
Let \(\mathcal O\subseteq\mathbb R_+^2\) be relatively open and set
\(\tau_{\mathcal O^c}:=\inf\{t\ge0:X_t\notin\mathcal O\}\). Let
\(u:\mathbb R_+^2\to\mathbb R\) be continuous and of polynomial growth.
Assume that:
\begin{enumerate}
\item[\textup{(V1)}]
\(u\ge G\) on \(\mathbb R_+^2\).

\item[\textup{(V2)}]
\(u=G\) on \(\mathbb R_+^2\setminus\mathcal O\).

\item[\textup{(V3)}]
For some \(p>2\), \(u\in W^{2,p}_{\rm loc}(\mathcal O\cap\mathbb R_{++}^2)\)
and \((\mathcal L-r)u-c=0\) a.e. in
\(\mathcal O\cap\mathbb R_{++}^2\).

\item[\textup{(V4)}]
The distribution \(\mu_u:=(\mathcal L-r)u-c\) extends to a signed Radon
measure on \(\mathbb R_{++}^2\) satisfying
\[
   \mu_u\le0
   \quad\text{in the sense of measures, i.e.}\quad
   \langle \mu_u,\phi\rangle\le0
   \quad\forall\,\phi\in C_c^\infty(\mathbb R_{++}^2),\ \phi\ge0 .
\]

\item[\textup{(V5)}]
\(u\) is It\^o--Krylov--Tanaka admissible in the sense of
Assumption~\ref{ass:ito-krylov-admissibility}.
\end{enumerate}
Then \(u\ge V\) on \(\mathbb R_+^2\). If, in addition,
\(\tau_{\mathcal O^c}<\infty\) \(\mathbb P_x\)-a.s. for every
\(x\in\mathbb R_+^2\), the stopped formula is exact on
\([0,\tau_{\mathcal O^c}]\), and the terminal passage to
\(\tau_{\mathcal O^c}\) is justified, then \(u=V\) and
\(\tau_{\mathcal O^c}\) is optimal.
\end{theorem}

\begin{remark}[Diagonal compatibility obstruction in the verification theorem]
\label{rem:diagonal-compatibility-obstruction}
The hypotheses \(u=G\) on \(\mathbb R_+^2\setminus\mathcal O\) and
\(\mu_u:=(\mathcal L-r)u-c\le0\) as a signed measure impose a genuine
restriction along the kink set \(\Delta=\{x_1=\alpha x_2\}\). Indeed, on the
interior stopping side \(\mathbb R_{++}^2\setminus\mathcal O\), the contact
condition gives \(u=G\), hence
   $\mu_u
   =
   (\mathcal L-r)u-c
   =
   (\mathcal L-r)G-c
   =
   -\Gamma$.
Since the diagonal part of the stopping-gain measure is
\[
   \Gamma^\Delta(dx)
   =
   -\frac{q(x)}{2\sqrt{1+\alpha^2}}\,\sigma_\Delta(dx),
   \qquad
   q(x)=n^\top a(x)n,\quad n=(1,-\alpha),
\]
the diagonal singular part of \(\mu_u\) on the contact side is
\[
   \mu_u^\Delta(dx)
   =
   -\Gamma^\Delta(dx)
   =
   \frac{q(x)}{2\sqrt{1+\alpha^2}}\,\sigma_\Delta(dx)\ge0.
\]
However, the verification theorem assumes
   $\mu_u\le0$ as a signed measure on $\mathbb R_{++}^2$.
Therefore, if \(E:=(\mathbb R_{++}^2\setminus\mathcal O)\cap\Delta\), the
positive diagonal contribution must vanish:
\[
   0
   \ge
   \mu_u^\Delta(E)
   =
   \frac{1}{2\sqrt{1+\alpha^2}}
   \int_E q(x)\,\sigma_\Delta(dx)
   \ge0,
\]
and hence
\begin{equation}
\label{eq:diagonal-compatibility-condition}
   \int_{(\mathbb R_{++}^2\setminus\mathcal O)\cap\Delta}
      q(x)\,\sigma_\Delta(dx)=0 .
\end{equation}
Under local uniform ellipticity in \(\mathbb R_{++}^2\), one has, locally away
from the coordinate axes,
\[
   q(x)=n^\top a(x)n>0,
   \qquad x\in\mathbb R_{++}^2,\quad n=(1,-\alpha)\ne0.
\]
Consequently, on every compact subset \(K\Subset\mathbb R_{++}^2\),
\[
   \int_{K\cap(\mathbb R_{++}^2\setminus\mathcal O)\cap\Delta}
      q(x)\,\sigma_\Delta(dx)=0
   \quad\Longleftrightarrow\quad
   \sigma_\Delta\!\left(
      K\cap(\mathbb R_{++}^2\setminus\mathcal O)\cap\Delta
   \right)=0 .
\]
Thus, away from the axes, the verification theorem is compatible with the
diagonal singular measure only when the contact side
\(\mathbb R_{++}^2\setminus\mathcal O\) intersects \(\Delta\) in zero
\(\sigma_\Delta\)-measure. Equivalently, the kink set \(\Delta\) must lie in
the continuation region up to a \(q\,\sigma_\Delta\)-null set.
This condition is not an additional free-boundary theorem. It is the
measure-theoretic compatibility requirement forced by the identity
   $(\mathcal L-r)G-c=-\Gamma$
and by the nonpositive diagonal component of \(\Gamma\). Without
\eqref{eq:diagonal-compatibility-condition}, the assumptions
\(u=G\) on the stopping side and \(\mu_u\le0\) would contradict the positive
singular part
\[
   \mu_u^\Delta(dx)
   =
   \frac{q(x)}{2\sqrt{1+\alpha^2}}\,\sigma_\Delta(dx).
\]
\end{remark}

\begin{remark}[No regularity is inferred from viscosity status]
\label{rem:main-no-hidden-regularity}
Theorem~\ref{thm:verification-measure} is deliberately a verification theorem,
not a regularity theorem. The hypotheses \(u\in W^{2,p}_{\rm loc}\), the
signed-measure extension of \((\mathcal L-r)u-c\), the reflected boundary
control, and the stopped It\^o--Krylov--Tanaka formula are substantive
assumptions. They are not consequences of being a continuous viscosity
supersolution.
\end{remark}

\begin{remark}[Compatibility of \textup{(V2)} and \textup{(V4)} on the diagonal]
\label{rem:main-V2-V4-compatibility}
Hypotheses \textup{(V2)} and \textup{(V4)} interact on \(\Delta\). On
\(\mathbb R_{++}^2\setminus\mathcal O\) one has \(u=G\), hence
\(\mu_u=(\mathcal L-r)G-c=-\Gamma\), whose singular part is
\(-\Gamma^\Delta(dx)=\dfrac{q(x)}{2\sqrt{1+\alpha^2}}\,\sigma_\Delta(dx)\ge0\).
Since \textup{(V4)} requires \(\mu_u\le0\) as a measure, the two hypotheses are
jointly satisfiable on a set that charges \(\Delta\) only if
   $\sigma_\Delta\bigl((\mathbb R_{++}^2\setminus\mathcal O)\cap\Delta\bigr)=0$ (equivalently  $q\equiv0$ there).
In other words, Theorem~\ref{thm:verification-measure} applies to
configurations in which the diagonal lies in the continuation set \(\mathcal O\)
up to a \(\sigma_\Delta\)-null set. In that case the diagonal surface measure
contributes only to the killed potential \(R_r^{\mathcal C}\Gamma\) of
Section~\ref{sec:main-killed-resolvent}, and not to the superharmonic defect on
the stopping set. This is consistent with
Remark~\ref{rem:Gamma-measure-warning}.
\end{remark}

\subsection{Conditional geometry of the stopping set}
\label{sec:main-geometry}

The next result identifies when the stopping set is an epigraph. This is not a
probabilistic monotonicity theorem. It is a deterministic geometric consequence
of monotonicity of \(H:=V-G\).

\begin{assumption}[Monotonicity of the stopping advantage]
\label{ass:H-monotonicity}
Let \(H:=V-G\). We impose the following structural conditions when the epigraph
result is used.
\begin{enumerate}
\item[\textup{(M1)}]
For every \(x_1\ge0\), the map \(x_2\mapsto H(x_1,x_2)\) is nonincreasing on
\([0,\infty)\).

\item[\textup{(M2)}]
Every vertical stopping section is nonempty:
\[
   \mathcal D(x_1):=\{x_2\ge0:(x_1,x_2)\in\mathcal D\}\ne\varnothing,
   \qquad x_1\ge0.
\]

\item[\textup{(M3)}]
When monotonicity of the boundary is required, we additionally assume that
\(x_1\mapsto H(x_1,x_2)\) is monotone on \([0,\infty)\), with a direction
independent of \(x_2\): either all horizontal sections are nonincreasing or all
horizontal sections are nondecreasing.
\end{enumerate}
\end{assumption}

\begin{theorem}[Conditional epigraph structure]
\label{thm:conditional-epigraph}
Assume \(V\) is continuous and Assumption~\ref{ass:H-monotonicity}
\textup{(M1)}--\textup{(M2)} holds. Define
   $b(x_1):=\inf\{x_2\ge0:(x_1,x_2)\in\mathcal D\}$,
   $x_1\ge0$.
Then \(b\) is well defined with values in \([0,\infty)\), and
   $\mathcal D
   =
   \{(x_1,x_2)\in\mathbb R_+^2:x_2\ge b(x_1)\}$,
   $\mathcal C
   =
   \{(x_1,x_2)\in\mathbb R_+^2:x_2<b(x_1)\}$.
If, in addition, Assumption~\ref{ass:H-monotonicity}\textup{(M3)} holds, then
\(b\) is monotone: if \(H\) is horizontally nonincreasing, then \(b\) is
nonincreasing; if \(H\) is horizontally nondecreasing, then \(b\) is
nondecreasing.
\end{theorem}

\begin{remark}[Logical status of the graph theorem]
\label{rem:main-graph-status}
Theorem~\ref{thm:conditional-epigraph} assumes monotonicity of the stopping
advantage \(H=V-G\). It does not assert that this monotonicity follows from
order preservation of the reflected diffusion, from monotonicity of \(V\), or
from pointwise monotonicity of the infinitesimal stopping gain. In concrete
models, Assumption~\ref{ass:H-monotonicity} must be verified separately.
\end{remark}

\begin{assumption}[Local regularity of an epigraph boundary]
\label{ass:boundary-regularity-main}
Whenever boundary traces, killed Green kernels, or non-tangential limits are
used, the boundary \(b\) in Theorem~\ref{thm:conditional-epigraph} is assumed to
satisfy \(b\in C((0,\infty))\cap W^{1,\infty}_{\rm loc}((0,\infty))\), and
for every compact interval \(I\Subset(0,\infty)\),
\(\inf_{x_1\in I}b(x_1)>0\).
\end{assumption}

\subsection{The singular diagonal stopping-gain measure}
\label{sec:main-singular-measure}

The max payoff is affine away from \(\Delta=\{x_1=\alpha x_2\}\), but its
second derivative across \(\Delta\) is a surface measure. The following result
gives the precise decomposition of \(\Gamma=c+rG-\mathcal LG\).

\begin{theorem}[Diagonal singular measure generated by the max payoff]
\label{thm:diagonal-singular-measure}
Let \(G(x_1,x_2)=x_1\vee\alpha x_2\), let
\(\Delta:=\{x\in\mathbb R_+^2:x_1=\alpha x_2\}\), set
\(n:=(1,-\alpha)\), and define \(q(x):=n^\top a(x)n\). Then the
stopping-gain measure \(\Gamma=c+rG-\mathcal LG\) decomposes as
   $\Gamma=\Gamma^{\rm ac}\,dx+\Gamma^\Delta$,
where, on \(\mathcal R_1:=\{x_1>\alpha x_2\}\) and
\(\mathcal R_2:=\{x_1<\alpha x_2\}\),
\[
   \Gamma^{\rm ac}(x)
   =
   \begin{cases}
   c(x)+rx_1-\mu_1(x), & x\in\mathcal R_1,\\[1mm]
   c(x)+r\alpha x_2-\alpha\mu_2(x), & x\in\mathcal R_2,
   \end{cases}
\]
and the singular diagonal component is
\[
   \Gamma^\Delta(dx)
   =
   -\dfrac{n^\top a(x)n}{2\sqrt{1+\alpha^2}}\,\sigma_\Delta(dx)
   =
   -\dfrac{q(x)}{2\sqrt{1+\alpha^2}}\,\sigma_\Delta(dx).
\]
Equivalently, for every bounded Borel function \(F\),
\[
   \int_{\mathbb R_+^2}F(x)\,\Gamma^\Delta(dx)
   =
   -\frac{1}{2\sqrt{1+\alpha^2}}
   \int_\Delta F(x)q(x)\,\sigma_\Delta(dx).
\]
\end{theorem}

\begin{remark}[Sign of the diagonal component]
\label{rem:main-diagonal-sign}
Because \(a(x)\) is nonnegative definite, \(q(x)=n^\top a(x)n\ge0\), and hence
\(\Gamma^\Delta\le0\) as a signed measure. Thus the usual pointwise
stopping-gain sign condition cannot be imposed literally on the full measure
\(\Gamma\) when the stopping set intersects \(\Delta\) in positive
one-dimensional measure. Any sign condition must specify whether it concerns
\(\Gamma^{\rm ac}\) away from \(\Delta\), the full signed measure, or a
decomposition in which the diagonal part is handled separately.
\end{remark}

\subsection{Killed-resolvent representation}
\label{sec:main-killed-resolvent}

The central potential-theoretic correction is that the resolvent must be killed
at the first entry into the stopping set. The unrestricted reflected resolvent
counts post-stopping occupation and is generally not the correct object.

For a signed smooth measure \(\nu=\nu^+-\nu^-\) such that both killed potentials
are finite, define
\[
   R_r^{\mathcal C}\nu(x)
   :=
   \mathbb E_x\left[
      \int_0^{\tau_{\mathcal D}}e^{-rs}\,dA_s^\nu
   \right],
   \qquad
   A^\nu:=A^{\nu^+}-A^{\nu^-},
\]
where \(\tau_{\mathcal D}:=\inf\{t\ge0:X_t\in\mathcal D\}\). If
\(\nu=f\,dx\), then
\[
   R_r^{\mathcal C}f(x)
   =
   \mathbb E_x\left[
      \int_0^{\tau_{\mathcal D}}e^{-rs}f(X_s)\,ds
   \right].
\]
By contrast, the unrestricted reflected resolvent is
\[
   R_r^{\rm R}f(x)
   :=
   \mathbb E_x\left[
      \int_0^\infty e^{-rs}f(X_s)\,ds
   \right].
\]

\begin{assumption}[Admissibility of the killed stopping-gain potential]
\label{ass:killed-potential-main}
The signed measure \(\Gamma=\Gamma^+-\Gamma^-\) has positive and negative
parts that are smooth measures for \(X\), and
   $R_r^{\mathcal C}\Gamma^+(x)+R_r^{\mathcal C}\Gamma^-(x)<\infty$,
   $x\in\mathbb R_+^2$.
The associated signed additive functional is
\(A^\Gamma:=A^{\Gamma^+}-A^{\Gamma^-}\), and the stopped
It\^o--Tanaka identity for \(G(X)\) holds up to
\(\tau_{\mathcal D}\) with the required localization and uniform integrability.
\end{assumption}

\begin{theorem}[Killed-resolvent representation of the value]
\label{thm:killed-resolvent-representation}
Assume that \(\tau_{\mathcal D}\) is optimal and that
Assumption~\ref{ass:killed-potential-main} holds. Then, for every
\(x\in\mathbb R_+^2\),
\[
   V(x)
   =
   G(x)-R_r^{\mathcal C}\Gamma(x)
   =
   G(x)
   -
   \mathbb E_x\left[
      \int_0^{\tau_{\mathcal D}}e^{-rs}\,dA_s^\Gamma
   \right].
\]
In general,
   $R_r^{\mathcal C}\Gamma
   \ne
   R_r^{\rm R}(\Gamma\mathbf 1_{\mathcal C})$,
because the unrestricted reflected process continues after
\(\tau_{\mathcal D}\) and may later re-enter \(\mathcal C\).
\end{theorem}

\begin{remark}[Killed versus unrestricted reflected resolvent]
\label{rem:killed-versus-unrestricted-resolvent-identity}
Let \(\nu\) be a positive smooth measure for \(X\), and assume that all
potentials below are finite. The unrestricted reflected resolvent of
\(\nu\mathbf 1_{\mathcal C}\) is
\[
   R_r^{\rm R}(\nu\mathbf 1_{\mathcal C})(x)
   :=
   \mathbb E_x\!\left[
      \int_0^\infty e^{-rs}\mathbf 1_{\mathcal C}(X_s)\,dA_s^\nu
   \right].
\]
Splitting the integral at
\(\tau_{\mathcal D}:=\inf\{t\ge0:X_t\in\mathcal D\}\) gives
\begin{equation}
\label{eq:unrestricted-split-at-tauD}
   R_r^{\rm R}(\nu\mathbf 1_{\mathcal C})(x)=
   \mathbb E_x\!\left[
      \int_0^{\tau_{\mathcal D}}
         e^{-rs}\mathbf 1_{\mathcal C}(X_s)\,dA_s^\nu
   \right] +
   \mathbb E_x\!\left[
      \int_{\tau_{\mathcal D}}^\infty
         e^{-rs}\mathbf 1_{\mathcal C}(X_s)\,dA_s^\nu
   \right].
\end{equation}
Since \(X_s\in\mathcal C\) for \(0\le s<\tau_{\mathcal D}\), the first term is
exactly the killed potential:
\[
   \mathbb E_x\!\left[
      \int_0^{\tau_{\mathcal D}}
         e^{-rs}\mathbf 1_{\mathcal C}(X_s)\,dA_s^\nu
   \right]
   =
   \mathbb E_x\!\left[
      \int_0^{\tau_{\mathcal D}}e^{-rs}\,dA_s^\nu
   \right]
   =
   R_r^{\mathcal C}\nu(x).
\]
For the second term, use the additivity of \(A^\nu\), the time shift
\(\theta_t\), and the strong Markov property at \(\tau_{\mathcal D}\). On
\(\{\tau_{\mathcal D}<\infty\}\),
\[
   \int_{\tau_{\mathcal D}}^\infty
      e^{-rs}\mathbf 1_{\mathcal C}(X_s)\,dA_s^\nu
   =
   e^{-r\tau_{\mathcal D}}
   \int_0^\infty
      e^{-ru}\mathbf 1_{\mathcal C}(X_{\tau_{\mathcal D}+u})\,
      d\big(A_{\tau_{\mathcal D}+u}^\nu-A_{\tau_{\mathcal D}}^\nu\big),
\]
and hence
\[
   \mathbb E_x\!\left[
      \int_{\tau_{\mathcal D}}^\infty
         e^{-rs}\mathbf 1_{\mathcal C}(X_s)\,dA_s^\nu
   \right]   =
   \mathbb E_x\!\left[
      e^{-r\tau_{\mathcal D}}
      \mathbb E_{X_{\tau_{\mathcal D}}}\!\left[
         \int_0^\infty
            e^{-ru}\mathbf 1_{\mathcal C}(X_u)\,dA_u^\nu
      \right]
   \right]   =
   \mathbb E_x\!\left[
      e^{-r\tau_{\mathcal D}}
      R_r^{\rm R}(\nu\mathbf 1_{\mathcal C})(X_{\tau_{\mathcal D}})
   \right].
\]
Combining this identity with \eqref{eq:unrestricted-split-at-tauD} gives the
resolvent decomposition
   $R_r^{\rm R}(\nu\mathbf 1_{\mathcal C})(x)
   =
   R_r^{\mathcal C}\nu(x)
   +
   \mathbb E_x\!\left[
      e^{-r\tau_{\mathcal D}}
      R_r^{\rm R}(\nu\mathbf 1_{\mathcal C})(X_{\tau_{\mathcal D}})
   \right]$.
Therefore,
   $R_r^{\mathcal C}\nu(x)
   =
   R_r^{\rm R}(\nu\mathbf 1_{\mathcal C})(x)$
can hold only if the post-stopping contribution vanishes:
\begin{equation}
\label{eq:post-stopping-term-vanishes}
   \mathbb E_x\!\left[
      e^{-r\tau_{\mathcal D}}
      R_r^{\rm R}(\nu\mathbf 1_{\mathcal C})(X_{\tau_{\mathcal D}})
   \right]=0.
\end{equation}
In general there is no reason for \eqref{eq:post-stopping-term-vanishes} to
hold. Even if \(X_{\tau_{\mathcal D}}\in\mathcal D\), the unrestricted
reflected process is not killed at \(\tau_{\mathcal D}\); it continues to
evolve and may later re-enter \(\mathcal C\). 
Hence $R_r^{\mathcal C}\nu
   \ne
   R_r^{\rm R}(\nu\mathbf 1_{\mathcal C})$ in general.
For a signed admissible measure \(\nu=\nu^+-\nu^-\), the same computation is
applied separately to \(\nu^+\) and \(\nu^-\), provided
   $R_r^{\mathcal C}\nu^+(x)+R_r^{\mathcal C}\nu^-(x)
   +R_r^{\rm R}(\nu^+\mathbf 1_{\mathcal C})(x)
   +R_r^{\rm R}(\nu^-\mathbf 1_{\mathcal C})(x)<\infty$.
Thus every identity above is understood through the Jordan decomposition
   $R_r^{\mathcal C}\nu
   =
   R_r^{\mathcal C}\nu^+-R_r^{\mathcal C}\nu^-$,
   $R_r^{\rm R}(\nu\mathbf 1_{\mathcal C})
   =
   R_r^{\rm R}(\nu^+\mathbf 1_{\mathcal C})
   -
   R_r^{\rm R}(\nu^-\mathbf 1_{\mathcal C})$.
This is the algebraic reason why the correct value representation is
   $V=G-R_r^{\mathcal C}\Gamma$,
not
   $V=G-R_r^{\rm R}(\Gamma\mathbf 1_{\mathcal C})$.
\end{remark}

\subsection{Boundary trace condition}
\label{sec:main-boundary-trace}

The killed-resolvent representation implies a continuation-side trace
condition on the free boundary. This is a trace statement, not a substantive
equation obtained by starting the killed process at a boundary point. If
\(x\in\mathcal D\), then \(\tau_{\mathcal D}=0\), so the killed potential
vanishes trivially.

\begin{proposition}[Continuation-side trace condition]
\label{prop:boundary-trace}
Assume the hypotheses of Theorem~\ref{thm:killed-resolvent-representation}.
Suppose, in addition, that \(\mathcal D\) has the epigraph representation
\(\mathcal D=\{(x_1,x_2)\in\mathbb R_+^2:x_2\ge b(x_1)\}\), and let
\(z_b(x_1):=(x_1,b(x_1))\). At every boundary point where the continuation-side
trace exists,
\[
   \lim_{\substack{x\to z_b(x_1)\\ x\in\mathcal C}}
   R_r^{\mathcal C}\Gamma(x)=0.
\]
If non-tangential traces are used, the condition is interpreted as
\[
   \lim_{\substack{x\to z_b(x_1)\\ x\in\mathcal C\ {\rm n.t.}}}
   R_r^{\mathcal C}\Gamma(x)=0 .
\]
\end{proposition}

\begin{remark}[Parameterized Green-kernel form of the diagonal potential]
\label{rem:green-kernel-diagonal-integral}
Assume that the reflected diffusion killed on
\(\mathcal D\) admits a killed \(r\)-Green kernel
\(G_r^{\mathcal C}(x,y)\), so that, for sufficiently integrable densities
\(f\),
\[
   R_r^{\mathcal C}f(x)
   =
   \int_{\mathcal C}G_r^{\mathcal C}(x,y)f(y)\,dy .
\]
Then the absolutely continuous part of the stopping-gain potential is
\[
   R_r^{\mathcal C}\Gamma^{\rm ac}(x)
   =
   \int_{\mathcal C}G_r^{\mathcal C}(x,y)\Gamma^{\rm ac}(y)\,dy .
\]
The diagonal part can be written as an explicit one-dimensional integral. Since
   $\Delta=\{(\alpha s,s):s\ge0\}$,
   $d\sigma_\Delta=\sqrt{1+\alpha^2}\,ds$,
and
\[
   \Gamma^\Delta(dz)
   =
   -\frac{q(z)}{2\sqrt{1+\alpha^2}}\,\sigma_\Delta(dz),
\]
we obtain, for \(x\in\mathcal C\),
\[
\begin{aligned}
   R_r^{\mathcal C}\Gamma^\Delta(x)
   &=
   \int_{\Delta\cap\mathcal C}
      G_r^{\mathcal C}(x,z)\,\Gamma^\Delta(dz) =
   -\int_{\Delta\cap\mathcal C}
      G_r^{\mathcal C}(x,z)
      \frac{q(z)}{2\sqrt{1+\alpha^2}}\,\sigma_\Delta(dz)  \\
   &=
   -\frac12
   \int_{\{s\ge0:(\alpha s,s)\in\mathcal C\}}
      G_r^{\mathcal C}(x,(\alpha s,s))q(\alpha s,s)\,ds .
\end{aligned}
\]
Hence the full killed-potential representation becomes
\[
   R_r^{\mathcal C}\Gamma(x)=
   \int_{\mathcal C}G_r^{\mathcal C}(x,y)\Gamma^{\rm ac}(y)\,dy 
   -\frac12
   \int_{\{s\ge0:(\alpha s,s)\in\mathcal C\}}
      G_r^{\mathcal C}(x,(\alpha s,s))q(\alpha s,s)\,ds .
\]
Using the explicit formula for \(\Gamma^{\rm ac}\), this may also be written as
\[
\begin{aligned}
   R_r^{\mathcal C}\Gamma(x)
   &=
   \int_{\mathcal C\cap\mathcal R_1}
      G_r^{\mathcal C}(x,y)\big(c(y)+ry_1-\mu_1(y)\big)\,dy +
   \int_{\mathcal C\cap\mathcal R_2}
      G_r^{\mathcal C}(x,y)\big(c(y)+r\alpha y_2-\alpha\mu_2(y)\big)\,dy  \\
   &\quad
   -\frac12
   \int_{\{s\ge0:(\alpha s,s)\in\mathcal C\}}
      G_r^{\mathcal C}(x,(\alpha s,s))q(\alpha s,s)\,ds .
\end{aligned}
\]
Consequently, at a regular boundary point \(z_b(x_1)=(x_1,b(x_1))\), the
continuation-side boundary trace condition
\[
   \lim_{\substack{x\to z_b(x_1)\\x\in\mathcal C}}
   R_r^{\mathcal C}\Gamma(x)=0
\]
takes the concrete form
\[
   0=
   \lim_{\substack{x\to z_b(x_1)\\x\in\mathcal C}}
   \bigg[
      \int_{\mathcal C}G_r^{\mathcal C}(x,y)\Gamma^{\rm ac}(y)\,dy  
      -\frac12
      \int_{\{s\ge0:(\alpha s,s)\in\mathcal C\}}
         G_r^{\mathcal C}(x,(\alpha s,s))q(\alpha s,s)\,ds
   \bigg].
\]
Equivalently, after expanding \(\Gamma^{\rm ac}\),
\[
\begin{aligned}
   0
   &=
   \lim_{\substack{x\to z_b(x_1)\\x\in\mathcal C}}
   \bigg[
      \int_{\mathcal C\cap\mathcal R_1}
         G_r^{\mathcal C}(x,y)\big(c(y)+ry_1-\mu_1(y)\big)\,dy +
      \int_{\mathcal C\cap\mathcal R_2}
         G_r^{\mathcal C}(x,y)\big(c(y)+r\alpha y_2-\alpha\mu_2(y)\big)\,dy \\
   &\qquad\qquad-
      \frac12
      \int_{\{s\ge0:(\alpha s,s)\in\mathcal C\}}
         G_r^{\mathcal C}(x,(\alpha s,s))q(\alpha s,s)\,ds
   \bigg].
\end{aligned}
\]
The last integral is the explicit contribution of the kink. Thus the
boundary-trace equation is not merely a volume equation over \(\mathcal C\);
it contains a one-dimensional diagonal correction generated by the singular
part of \(c+rG-\mathcal LG\).
\end{remark}

\subsection{Candidate-boundary verification}
\label{sec:main-candidate-verification}

The final result explains how a proposed boundary should be verified. Solving
a boundary-trace equation alone is not enough. One must verify majorization,
contact, the continuation equation, reflected boundary behavior, global
measure-superharmonicity, and the relevant integrability assumptions.

Let \(h:\mathbb R_+\to\mathbb R_+\) be a candidate boundary and define
\[
   \mathcal D_h:=\{(x_1,x_2)\in\mathbb R_+^2:x_2\ge h(x_1)\},
   \qquad
   \mathcal C_h:=\mathbb R_+^2\setminus\mathcal D_h,
   \qquad
   \tau_h:=\inf\{t\ge0:X_t\in\mathcal D_h\}.
\]
For an admissible signed measure \(\nu\), set
\[
   R_r^{\mathcal C_h}\nu(x)
   :=
   \mathbb E_x\left[
      \int_0^{\tau_h}e^{-rs}\,dA_s^\nu
   \right],
\]
and define the candidate value
   $U_h(x):=G(x)-R_r^{\mathcal C_h}\Gamma(x)$.

\begin{theorem}[Candidate-boundary verification]
\label{thm:candidate-verification}
Let \(h:\mathbb R_+\to\mathbb R_+\) be a candidate boundary and let
\(U_h:=G-R_r^{\mathcal C_h}\Gamma\). Assume:
\begin{enumerate}
\item[\textup{(C1)}]
\(U_h\) is continuous and of polynomial growth.

\item[\textup{(C2)}]
\(U_h\ge G\) on \(\mathbb R_+^2\).

\item[\textup{(C3)}]
\(U_h=G\) on \(\mathcal D_h\). Under the convention \(\tau_h=0\) on
\(\mathcal D_h\), this is automatic whenever \(R_r^{\mathcal C_h}\Gamma\) is
defined by the killed potential above.

\item[\textup{(C4)}]
In \(\mathcal C_h\cap\mathbb R_{++}^2\), \(U_h\) satisfies the continuation
equation $(\mathcal L-r)U_h-c=0$
in the weak, viscosity, or It\^o--Krylov--Tanaka sense required by
Theorem~\ref{thm:verification-measure}.

\item[\textup{(C5)}]
\(U_h\) satisfies the normal-reflection condition
\(\partial_iU_h=0\) on \(\{x_i=0\}\), \(i=1,2\), in the relevant trace or
reflected-viscosity sense.

\item[\textup{(C6)}]
\(U_h\) is globally measure-superharmonic relative to the running cost:
   $(\mathcal L-r)U_h-c\le0$ as a signed measure on $\mathbb R_{++}^2$.

\item[\textup{(C7)}]
\(U_h\) is It\^o--Krylov--Tanaka admissible and satisfies the localization,
signed-additive-functional, and uniform-integrability assumptions required in
Theorem~\ref{thm:verification-measure}.
\end{enumerate}
Then \(U_h=V\) on \(\mathbb R_+^2\), and \(\tau_h\) is optimal. If, in
addition,
   $U_h>G$ on $\mathcal C_h$,
then \(\mathcal D_h=\mathcal D\). Consequently, if the true stopping set is
known to be an epigraph
\(\mathcal D=\{(x_1,x_2):x_2\ge b(x_1)\}\), then \(h=b\) at every point where
both graph representatives are defined with the same convention.
\end{theorem}

\begin{remark}[No uniqueness from the trace condition alone]
\label{rem:main-no-fredholm-uniqueness}
The trace condition
\[
   \lim_{\substack{x\to z_h(x_1)\\x\in\mathcal C_h}}
   R_r^{\mathcal C_h}\Gamma(x)=0
\]
does not by itself prove uniqueness of \(h\), nor does it prove
\(U_h=G\) exactly on \(\mathcal D_h\). Boundary uniqueness follows only after
the verification argument above, and identification of \(h\) with the canonical
boundary additionally requires strict continuation, \(U_h>G\) on
\(\mathcal C_h\).
\end{remark}

\begin{remark}[Diagonal obstruction in candidate-boundary verification]
\label{rem:candidate-diagonal-obstruction}
Condition \textup{(C6)} in Theorem~\ref{thm:candidate-verification} is a
global signed-measure condition and is not automatic from the contact condition
\(U_h=G\) on \(\mathcal D_h\). The obstruction occurs precisely on the kink set
   $\Delta=\{x\in\mathbb R_+^2:x_1=\alpha x_2\}$.
Indeed, on the contact side \(\mathcal D_h\), if \(U_h=G\), then formally
   $(\mathcal L-r)U_h-c
   =
   (\mathcal L-r)G-c
   =
   -\Gamma$.
Since the stopping-gain measure satisfies
\[
   \Gamma^\Delta(dx)
   =
   -\frac{q(x)}{2\sqrt{1+\alpha^2}}\,\sigma_\Delta(dx),
   \qquad
   q(x)=n^\top a(x)n,\quad n=(1,-\alpha),
\]
the diagonal singular component of the candidate superharmonicity measure is
\[
   \big[(\mathcal L-r)U_h-c\big]^\Delta(dx)
   =
   -\Gamma^\Delta(dx)
   =
   \frac{q(x)}{2\sqrt{1+\alpha^2}}\,\sigma_\Delta(dx)\ge0 .
\]
However, candidate verification requires
   $(\mathcal L-r)U_h-c\le0$
   as a signed measure on $\mathbb R_{++}^2$.
Therefore, setting \(E_h:=\mathcal D_h\cap\Delta\cap\mathbb R_{++}^2\), the
positive diagonal contribution must vanish:
\[
   0
   \ge
   \big[(\mathcal L-r)U_h-c\big]^\Delta(E_h)
   =
   \frac{1}{2\sqrt{1+\alpha^2}}
   \int_{E_h}q(x)\,\sigma_\Delta(dx)
   \ge0.
\]
Hence a candidate satisfying \textup{(C6)} must obey
\[
   \int_{\mathcal D_h\cap\Delta\cap\mathbb R_{++}^2}
      q(x)\,\sigma_\Delta(dx)=0 .
\]
Under local uniform ellipticity in \(\mathbb R_{++}^2\), \(q(x)=n^\top a(x)n>0\)
locally away from the coordinate axes because \(n=(1,-\alpha)\ne0\). Thus,
on every compact \(K\Subset\mathbb R_{++}^2\),
\[
   \int_{K\cap\mathcal D_h\cap\Delta}
      q(x)\,\sigma_\Delta(dx)=0
   \quad\Longleftrightarrow\quad
   \sigma_\Delta(K\cap\mathcal D_h\cap\Delta)=0 .
\]
In particular, a candidate stopping region that contains a positive
\(\sigma_\Delta\)-length portion of the diagonal in the interior generally
violates \textup{(C6)}.
Equivalently, a valid candidate must either keep the diagonal kink in the
continuation region up to a \(q\,\sigma_\Delta\)-null set, or construct
\(U_h\) so that the value is not merely the raw contact function \(G\) across a
positive-measure portion of \(\Delta\) in the measure-superharmonic sense. This
is why Theorem~\ref{thm:candidate-verification} requires global
measure-superharmonicity and does not validate a candidate boundary solely from
the trace equation
\[
   \lim_{\substack{x\to z_h(x_1)\\x\in\mathcal C_h}}
   R_r^{\mathcal C_h}\Gamma(x)=0 .
\]
\end{remark}

\section{Proofs}
\label{sec:proofs}

This section proves the results stated in Section~\ref{sec:main-results}. The
order follows the order of the theorem statements: verification,
conditional geometry, singular stopping-gain measure, killed-resolvent
representation, boundary trace, and candidate-boundary verification.

\subsection{Proof of the verification theorem}
\label{sec:proof-verification}

\begin{proof}[Proof of Theorem~\ref{thm:verification-measure}]
Fix \(x\in\mathbb R_+^2\) and let \(\tau\in\mathcal T\). Choose a localizing
sequence \((\tau_n)\) such that
   $\tau_n\le \tau$, $\tau_n\uparrow\tau$,
and all stopped martingale, boundary, additive-functional, payoff, and cost
terms below are integrable. By Assumption~\ref{ass:ito-krylov-admissibility},
with \(\mu_u:=(\mathcal L-r)u-c\), the stopped
It\^o--Krylov--Tanaka formula gives
\begin{equation}
\label{eq:proof-verification-ito}
   e^{-r\tau_n}u(X_{\tau_n})
   =
   u(x)
   +\int_0^{\tau_n}e^{-rs}c(X_s)\,ds
   +\int_0^{\tau_n}e^{-rs}\,dA_s^{\mu_u}
   +B_{\tau_n}^u
   +M_{\tau_n}^u .
\end{equation}
Here \(M^u\) is a stopped local martingale, \(B^u\) is the reflected boundary
contribution, and \(A^{\mu_u}=A^{\mu_u^+}-A^{\mu_u^-}\) is the signed
continuous additive functional associated with the signed measure
\(\mu_u=\mu_u^+-\mu_u^-\).

By assumption, \(\mu_u\le0\) in the sense of measures. Hence the signed
finite-variation term satisfies
\[
   \mathbb E_x\!\left[
      \int_0^{\tau_n}e^{-rs}\,dA_s^{\mu_u}
   \right]\le0.
\]
The reflected boundary contribution is also nonpositive in expectation,
\(\mathbb E_x[B_{\tau_n}^u]\le0\), and the stopped martingale satisfies
\(\mathbb E_x[M_{\tau_n}^u]=0\). Taking expectations in
\eqref{eq:proof-verification-ito} yields
\begin{equation}
\label{eq:proof-supermartingale-ineq}
   \mathbb E_x\left[
      e^{-r\tau_n}u(X_{\tau_n})
      -
      \int_0^{\tau_n}e^{-rs}c(X_s)\,ds
   \right]\le u(x).
\end{equation}
Since \(u\ge G\), we have
\[
   e^{-r\tau_n}G(X_{\tau_n})
   -
   \int_0^{\tau_n}e^{-rs}c(X_s)\,ds
   \le
   e^{-r\tau_n}u(X_{\tau_n})
   -
   \int_0^{\tau_n}e^{-rs}c(X_s)\,ds.
\]
Combining this with \eqref{eq:proof-supermartingale-ineq} gives
\[
   \mathbb E_x\left[
      e^{-r\tau_n}G(X_{\tau_n})
      -
      \int_0^{\tau_n}e^{-rs}c(X_s)\,ds
   \right]\le u(x).
\]
By the uniform-integrability and localization assumptions, passing to the limit
\(n\to\infty\) gives
\[
   \mathbb E_x\left[
      e^{-r\tau}G(X_\tau)
      -
      \int_0^\tau e^{-rs}c(X_s)\,ds
   \right]\le u(x).
\]
Taking the supremum over all \(\tau\in\mathcal T\) proves \(V(x)\le u(x)\).

It remains to prove equality under the additional assumptions. Apply the same
formula to \(u(X)\) on \([0,t\wedge\tau_{\mathcal O^c}]\). Since
\((\mathcal L-r)u-c=0\) a.e. in \(\mathcal O\cap\mathbb R_{++}^2\), no
measure-superharmonic defect is accumulated before \(\tau_{\mathcal O^c}\).
The exact stopped formula therefore gives
\[
   u(x)
   =
   \mathbb E_x\left[
      e^{-r(t\wedge\tau_{\mathcal O^c})}
      u(X_{t\wedge\tau_{\mathcal O^c}})
      -
      \int_0^{t\wedge\tau_{\mathcal O^c}}e^{-rs}c(X_s)\,ds
   \right].
\]
Letting \(t\to\infty\), using \(\tau_{\mathcal O^c}<\infty\) a.s.,
\(u=G\) on \(\mathbb R_+^2\setminus\mathcal O\), and the terminal
uniform-integrability hypotheses, we obtain
\[
   u(x)
   =
   \mathbb E_x\left[
      e^{-r\tau_{\mathcal O^c}}G(X_{\tau_{\mathcal O^c}})
      -
      \int_0^{\tau_{\mathcal O^c}}e^{-rs}c(X_s)\,ds
   \right]\le V(x).
\]
Together with \(V\le u\), this proves \(u=V\), and
\(\tau_{\mathcal O^c}\) is optimal.
\end{proof}

\subsection{Proof of the conditional epigraph theorem}
\label{sec:proof-epigraph}

\begin{proof}[Proof of Theorem~\ref{thm:conditional-epigraph}]
The proof is deterministic and uses only \(H\ge0\), continuity of \(H\), and
the monotonicity assumptions. Fix \(x_1\ge0\). By
Assumption~\ref{ass:H-monotonicity}\textup{(M2)}, the vertical section
\(\mathcal D(x_1)=\{x_2\ge0:H(x_1,x_2)=0\}\) is nonempty, so
\(b(x_1):=\inf\mathcal D(x_1)\) is finite. Let
\(z_2\in\mathcal D(x_1)\) and \(y_2\ge z_2\). Since \(H\ge0\),
\(H(x_1,z_2)=0\), and \(x_2\mapsto H(x_1,x_2)\) is nonincreasing, we have
   $0\le H(x_1,y_2)\le H(x_1,z_2)=0$.
Thus \(H(x_1,y_2)=0\), so \(y_2\in\mathcal D(x_1)\). Hence each vertical
section \(\mathcal D(x_1)\) is an upper interval.

Since \(V\) and \(G\) are continuous, \(H=V-G\) is continuous. Therefore
\(\mathcal D=\{H=0\}\) is closed, and each closed upper interval
\(\mathcal D(x_1)\) has the form \([b(x_1),\infty)\). Hence
   $\mathcal D
   =
   \{(x_1,x_2)\in\mathbb R_+^2:x_2\ge b(x_1)\}$,
   $\mathcal C
   =
   \{(x_1,x_2)\in\mathbb R_+^2:x_2<b(x_1)\}$.

Assume now that \(H\) is horizontally nonincreasing. Let \(0\le x_1\le y_1\)
and take \(x_2>b(x_1)\). Then \((x_1,x_2)\in\mathcal D\), so
\(H(x_1,x_2)=0\). Horizontal nonincreasingness gives
   $0\le H(y_1,x_2)\le H(x_1,x_2)=0$,
hence \((y_1,x_2)\in\mathcal D\) and \(b(y_1)\le x_2\). Letting
\(x_2\downarrow b(x_1)\) gives \(b(y_1)\le b(x_1)\). Therefore \(b\) is
nonincreasing. If \(H\) is horizontally nondecreasing, the same argument with
the roles of \(x_1\) and \(y_1\) reversed gives \(b(x_1)\le b(y_1)\), so
\(b\) is nondecreasing.
\end{proof}

\subsection{Proof of the singular-measure formula}
\label{sec:proof-singular-measure}

\begin{proof}[Proof of Theorem~\ref{thm:diagonal-singular-measure}]
Set
   $Y(x):=x_1-\alpha x_2$, $n:=\nabla Y=(1,-\alpha)$ and
   $q(x):=n^\top a(x)n$.
The proof has four steps.

First, write the max payoff as
\[
   G(x_1,x_2)
   =
   x_1\vee\alpha x_2
   =
   \frac{x_1+\alpha x_2+|x_1-\alpha x_2|}{2}
   =
   \frac{x_1+\alpha x_2+|Y(x)|}{2}.
\]
Thus \(G\) is affine on
\(\mathcal R_1=\{x_1>\alpha x_2\}\) and
\(\mathcal R_2=\{x_1<\alpha x_2\}\), and all singular second-order
contributions come from the \(|Y|\) term.

Second, apply Tanaka's formula to the continuous semimartingale
\(Y_t:=Y(X_t)=X_t^1-\alpha X_t^2\):
   $d|Y_t|
   =
   \operatorname{sgn}(Y_t)\,dY_t+dL_t^0(Y)$,
where \(L^0(Y)\) is the symmetric local time of \(Y\) at zero. Since
\[
   dY_t
   =
   \big(\mu_1(X_t)-\alpha\mu_2(X_t)\big)\,dt
   +n^\top\sigma(X_t)\,dW_t
   +dL_t^1-\alpha\,dL_t^2,
\]
the singular interior finite-variation term in \(dG(X_t)\) is
\(\frac12\,dL_t^0(Y)\). The boundary reflection terms are supported on
\(\{x_1=0\}\cup\{x_2=0\}\); they are treated separately through the reflected
boundary condition and the no-corner contribution assumption, and they do not
alter the interior diagonal measure on \(\Delta\).

Third, identify the quadratic variation of \(Y\). Since the martingale part of
\(Y\) is \(\int_0^t n^\top\sigma(X_s)\,dW_s\), we have
$d\langle Y\rangle_t
   =
   n^\top a(X_t)n\,dt
   =
   q(X_t)\,dt$.
The occupation-density formula gives
\[
L_t^0(Y)
   =
   \lim_{\varepsilon\downarrow0}
   \frac{1}{2\varepsilon}
   \int_0^t
      \mathbf 1_{\{|Y(X_s)|<\varepsilon\}}q(X_s)\,ds,
\]
or, in distributional notation,
   $dL_t^0(Y)=q(X_t)\delta_0(Y(X_t))\,dt$.

Fourth, convert the one-dimensional Dirac mass \(\delta_0(Y)\) into surface
measure on \(\Delta\). Since \(|\nabla Y|=|n|=\sqrt{1+\alpha^2}\), the co-area
identity yields
\[
   \delta_0(Y(x))\,dx
   =
   \frac{1}{|\nabla Y|}\,\sigma_\Delta(dx)
   =
   \frac{1}{\sqrt{1+\alpha^2}}\,\sigma_\Delta(dx).
\]
Therefore the singular part of \(\mathcal LG\) is
\[
   (\mathcal LG)^\Delta(dx)
   =
   \frac{q(x)}{2\sqrt{1+\alpha^2}}\,\sigma_\Delta(dx).
\]
Since \(\Gamma=c+rG-\mathcal LG\), the diagonal part of the stopping-gain
measure is
\[
   \Gamma^\Delta(dx)
   =
   -\frac{q(x)}{2\sqrt{1+\alpha^2}}\,\sigma_\Delta(dx)
   =
   -\frac{n^\top a(x)n}{2\sqrt{1+\alpha^2}}\,\sigma_\Delta(dx).
\]

It remains only to record the absolutely continuous part. On
\(\mathcal R_1\), \(G=x_1\), so \(\mathcal LG=\mu_1\) and
\(\Gamma^{\rm ac}=c+rx_1-\mu_1\). On \(\mathcal R_2\), \(G=\alpha x_2\), so
\(\mathcal LG=\alpha\mu_2\) and
\(\Gamma^{\rm ac}=c+r\alpha x_2-\alpha\mu_2\). Combining these two smooth-region
calculations with the diagonal component gives
   $\Gamma=\Gamma^{\rm ac}\,dx+\Gamma^\Delta$.
The stated integral identity against bounded Borel functions follows directly
from the definition of \(\Gamma^\Delta\).
\end{proof}

\subsection{Proof of the killed-resolvent representation}
\label{sec:proof-killed-resolvent}

\begin{proof}[Proof of Theorem~\ref{thm:killed-resolvent-representation}]
Fix \(x\in\mathbb R_+^2\). Since \(\tau_{\mathcal D}\) is optimal,
\begin{equation}
\label{eq:proof-optimality-tauD}
   V(x)
   =
   \mathbb E_x\left[
      e^{-r\tau_{\mathcal D}}G(X_{\tau_{\mathcal D}})
      -
      \int_0^{\tau_{\mathcal D}}e^{-rs}c(X_s)\,ds
   \right].
\end{equation}
Apply the stopped It\^o--Tanaka identity for \(G(X)\) on
\([0,t\wedge\tau_{\mathcal D}]\). With
\(\Gamma=c+rG-\mathcal LG\), and with the reflection contribution controlled
as assumed, this gives
\[
   e^{-r(t\wedge\tau_{\mathcal D})}
   G(X_{t\wedge\tau_{\mathcal D}})=
   G(x)
   +\int_0^{t\wedge\tau_{\mathcal D}}e^{-rs}c(X_s)\,ds 
   -\int_0^{t\wedge\tau_{\mathcal D}}e^{-rs}\,dA_s^\Gamma
   +M_{t\wedge\tau_{\mathcal D}} .
\]
After localization, the martingale term has zero expectation. Taking
expectations and rearranging gives
\[
   \mathbb E_x\left[
      e^{-r(t\wedge\tau_{\mathcal D})}
      G(X_{t\wedge\tau_{\mathcal D}})
      -
      \int_0^{t\wedge\tau_{\mathcal D}}e^{-rs}c(X_s)\,ds
   \right]
   =
   G(x)
   -
   \mathbb E_x\left[
      \int_0^{t\wedge\tau_{\mathcal D}}e^{-rs}\,dA_s^\Gamma
   \right].
\]
Letting \(t\to\infty\), and using the assumed uniform integrability of the
terminal, cost, and signed-additive-functional terms, yields
\[
   \mathbb E_x\left[
      e^{-r\tau_{\mathcal D}}G(X_{\tau_{\mathcal D}})
      -
      \int_0^{\tau_{\mathcal D}}e^{-rs}c(X_s)\,ds
   \right]
   =
   G(x)
   -
   \mathbb E_x\left[
      \int_0^{\tau_{\mathcal D}}e^{-rs}\,dA_s^\Gamma
   \right].
\]
Combining this identity with \eqref{eq:proof-optimality-tauD} gives
\[
   V(x)
   =
   G(x)
   -
   \mathbb E_x\left[
      \int_0^{\tau_{\mathcal D}}e^{-rs}\,dA_s^\Gamma
   \right]
   =
   G(x)-R_r^{\mathcal C}\Gamma(x).
\]
This proves the killed-resolvent representation.
\end{proof}

\begin{remark}[Killing versus unrestricted continuation]
\label{rem:proof-killing-versus-unrestricted}
The preceding proof stops the It\^o--Tanaka identity at
\(\tau_{\mathcal D}\). Thus the potential is accumulated only over the
pre-stopping interval \([0,\tau_{\mathcal D})\). Replacing it by
\(R_r^{\rm R}(\Gamma\mathbf 1_{\mathcal C})\) would integrate along the
unrestricted reflected process after \(\tau_{\mathcal D}\), and would count
possible later returns to \(\mathcal C\). This is not the potential appearing
in the optimal stopping representation.
\end{remark}

\subsection{Proof of the boundary trace condition}
\label{sec:proof-boundary-trace}

\begin{proof}[Proof of Proposition~\ref{prop:boundary-trace}]
For \(x\in\mathcal C\), Theorem~\ref{thm:killed-resolvent-representation}
gives $V(x)=G(x)-R_r^{\mathcal C}\Gamma(x)$,
hence $R_r^{\mathcal C}\Gamma(x)=G(x)-V(x)$.
Let \(z_b(x_1)=(x_1,b(x_1))\) be a boundary point at which the
continuation-side trace exists. Since \(V\) and \(G\) are continuous and
\(V=G\) on \(\partial\mathcal C\subseteq\mathcal D\), we have
\[
   \lim_{\substack{x\to z_b(x_1)\\x\in\mathcal C}}
   \big(G(x)-V(x)\big)=0.
\]
Therefore
\[
   \lim_{\substack{x\to z_b(x_1)\\x\in\mathcal C}}
   R_r^{\mathcal C}\Gamma(x)=0.
\]
The non-tangential version is identical, with the limit restricted to
non-tangential approach regions inside \(\mathcal C\).
\end{proof}

\subsection{Proof of the candidate-boundary verification theorem}
\label{sec:proof-candidate-verification}

\begin{proof}[Proof of Theorem~\ref{thm:candidate-verification}]
Let \(h:\mathbb R_+\to\mathbb R_+\) be a candidate boundary, and let
\(\mathcal D_h\), \(\mathcal C_h\), \(\tau_h\), and \(U_h\) be as in
Section~\ref{sec:main-candidate-verification}. By assumptions
\textup{(C1)}--\textup{(C7)}, the verification theorem
(Theorem~\ref{thm:verification-measure}) applies to \(U_h\) with
\(\mathcal O=\mathcal C_h\). Therefore
   $V(x)\le U_h(x)$, $x\in\mathbb R_+^2$.

It remains to prove the reverse inequality. Apply the stopped
It\^o--Krylov--Tanaka formula to \(U_h(X)\) on
\([0,t\wedge\tau_h]\). Since \((\mathcal L-r)U_h-c=0\) in
\(\mathcal C_h\), the reflected boundary condition removes the boundary
contribution, and the localization assumptions justify taking expectations,
we obtain
\[
   U_h(x)
   =
   \mathbb E_x\left[
      e^{-r(t\wedge\tau_h)}U_h(X_{t\wedge\tau_h})
      -
      \int_0^{t\wedge\tau_h}e^{-rs}c(X_s)\,ds
   \right].
\]
Letting \(t\to\infty\) and using the terminal integrability assumptions gives
\[
   U_h(x)
   =
   \mathbb E_x\left[
      e^{-r\tau_h}U_h(X_{\tau_h})
      -
      \int_0^{\tau_h}e^{-rs}c(X_s)\,ds
   \right].
\]
By the contact condition \(U_h=G\) on \(\mathcal D_h\), and by definition
\(X_{\tau_h}\in\mathcal D_h\) on \(\{\tau_h<\infty\}\), hence
\[
   U_h(x)
   =
   \mathbb E_x\left[
      e^{-r\tau_h}G(X_{\tau_h})
      -
      \int_0^{\tau_h}e^{-rs}c(X_s)\,ds
   \right]\le V(x).
\]
Together with \(V\le U_h\), this proves \(U_h=V\) on \(\mathbb R_+^2\), and
\(\tau_h\) is optimal.

Finally assume that \(U_h>G\) on \(\mathcal C_h\). Since \(U_h=G\) on
\(\mathcal D_h\), we have
   $\{U_h=G\}=\mathcal D_h$.
Because \(U_h=V\), it follows that
   $\mathcal D_h
   =
   \{U_h=G\}
   =
   \{V=G\}
   =
   \mathcal D$.
If both stopping sets are represented as epigraphs,
   $\mathcal D_h=\{(x_1,x_2):x_2\ge h(x_1)\}$,
   $\mathcal D=\{(x_1,x_2):x_2\ge b(x_1)\}$,
then equality of the epigraphs implies \(h=b\) at every point where both graph
representatives are defined with the same convention.
\end{proof}
\section{Conclusion}
\label{sec:conclusion}

This paper studied an infinite-horizon optimal stopping problem for a normally
reflected diffusion \(X\) in the quadrant \(\mathbb R_+^2\), with discounted
running cost \(c\) and max-type reward \(G(x_1,x_2)=x_1\vee\alpha x_2\). The
problem is constrained by reflection on the coordinate axes, genuinely
two-dimensional through its stopping region, and nonsmooth because \(G\) has a
kink on the diagonal \(\Delta=\{x_1=\alpha x_2\}\).

The main corrections are
\[
   \text{no hidden regularity},
   \qquad
   \text{measure-valued kink contribution},
   \qquad
   \text{killed-resolvent potential}.
\]
More precisely, viscosity supersolution status is not used as a substitute for
Sobolev, measure, or It\^o--Krylov--Tanaka admissibility; the stopping gain
\(\Gamma:=c+rG-\mathcal LG\) is treated as the signed measure
\(\Gamma=\Gamma^{\rm ac}\,dx+\Gamma^\Delta\), with
\(\Gamma^\Delta(dx)=-(q(x)/(2\sqrt{1+\alpha^2}))\,\sigma_\Delta(dx)\); and the
potential term is accumulated only before \(\tau_{\mathcal D}\), not along the
unrestricted reflected process.

Under explicit structural monotonicity and regularity assumptions, the stopping
set admits the epigraph representation
\(\mathcal D=\{(x_1,x_2)\in\mathbb R_+^2:x_2\ge b(x_1)\}\). Under the stated
measure-superharmonicity, admissibility, and optimality hypotheses, the value
has the killed-potential representation
\[
   V(x)=G(x)-R_r^{\mathcal C}\Gamma(x),
   \qquad
   R_r^{\mathcal C}\Gamma(x)
   =
   \mathbb E_x\!\left[
      \int_0^{\tau_{\mathcal D}}e^{-rs}\,dA_s^\Gamma
   \right].
\]
This representation is the correct pre-stopping analogue of a resolvent formula
and is generally different from the unrestricted expression
\(G-R_r^{\rm R}(\Gamma\mathbf 1_{\mathcal C})\).

Several directions remain open. One is to verify monotonicity of
\(H=V-G\) in concrete reflected diffusion models. Another is to prove
continuity, local Lipschitz regularity, or smooth fit of the free boundary
under additional analytic or probabilistic hypotheses. A third direction is to
develop numerical methods for the killed-resolvent boundary-trace condition
\(\lim_{x\to z_b,\,x\in\mathcal C}R_r^{\mathcal C}\Gamma(x)=0\). A fourth is
to extend the framework to oblique reflection, higher-dimensional orthants, and
payoffs of the form \(G(x)=\max_{1\le i\le d}\beta_i x_i\). The framework
developed here does not claim unconditional free-boundary regularity. Its
purpose is instead to provide a rigorous conditional structure in which the
singular payoff geometry, reflection, and stopping-time killing are all treated
explicitly.

\appendix
\section{Technical Appendix}
\label{sec:appendix}
\subsection{Dynamic programming and viscosity characterization}
\label{sec:appendix-dpp-viscosity}

This appendix records the dynamic programming principle, the lower
semicontinuity argument, the reflected-viscosity convention, and the viscosity
characterization of the value function. These facts are standard in optimal
stopping theory and viscosity-solution theory, but they are collected here to
keep the main text focused on the measure-valued stopping gain and the
killed-resolvent representation.

\begin{proposition}[Dynamic programming principle]
\label{prop:dpp}
Under Assumption~\ref{ass:reflected-diffusion}, assume that the payoff family
in \eqref{eq:value-function} satisfies the usual integrability and measurable
selection conditions for optimal stopping of the strong Markov process \(X\).
Then, for every bounded stopping time \(\theta\) and every
\(x\in\mathbb R_+^2\),
\begin{equation}
\label{eq:appendix-dpp}
\begin{aligned}
   V(x)
   =
   \sup_{\tau\in\mathcal T}
   \mathbb E_x\Bigg[
    -\int_0^{\tau\wedge\theta}e^{-rs}c(X_s)\,ds
      +e^{-r\tau}G(X_\tau)\mathbf 1_{\{\tau\le\theta\}}  
      +e^{-r\theta}V(X_\theta)\mathbf 1_{\{\tau>\theta\}}
   \Bigg].
\end{aligned}
\end{equation}
Equivalently, before the decision time \(\theta\), the controller either stops
at \(\tau\le\theta\) and receives \(G(X_\tau)\), or continues beyond \(\theta\)
and receives the continuation value \(V(X_\theta)\).
\end{proposition}

\begin{proof}
For any \(\tau\in\mathcal T\), split the payoff at \(\theta\):
\[
   e^{-r\tau}G(X_\tau)
   -
   \int_0^\tau e^{-rs}c(X_s)\,ds
   =
   \Big(
      e^{-r\tau}G(X_\tau)
      -
      \int_0^\tau e^{-rs}c(X_s)\,ds
   \Big)\mathbf 1_{\{\tau\le\theta\}}
\]
\[
   \quad
   +
   \Big(
      -\int_0^\theta e^{-rs}c(X_s)\,ds
      +e^{-r\theta}
        \big[
           e^{-r(\tau-\theta)}G(X_\tau)
           -
           \int_\theta^\tau e^{-r(s-\theta)}c(X_s)\,ds
        \big]
   \Big)\mathbf 1_{\{\tau>\theta\}} .
\]
On \(\{\tau>\theta\}\), the strong Markov property at \(\theta\) identifies
the conditional post-\(\theta\) optimization with \(V(X_\theta)\). This gives
the ``\(\le\)'' inequality in \eqref{eq:appendix-dpp}. The reverse inequality
is obtained by concatenating an arbitrary stopping rule before \(\theta\) with
\(\varepsilon\)-optimal stopping rules started from \(X_\theta\) after
\(\theta\), using the standard measurable selection argument and the assumed
integrability. Letting \(\varepsilon\downarrow0\) proves
\eqref{eq:appendix-dpp}.
\end{proof}

\begin{lemma}[Lower semicontinuity of the value]
\label{lem:value-lsc}
Assume, in addition to Assumption~\ref{ass:reflected-diffusion}, the following
local bounded-time approximation and uniform-integrability property: for every
\(x\in\mathbb R_+^2\) and \(\varepsilon>0\), there exists a bounded stopping
time \(\tau\le T\) such that
\[
   V(x)
   \le
   \mathbb E_x\left[
      e^{-r\tau}G(X_\tau)
      -
      \int_0^\tau e^{-rs}c(X_s)\,ds
   \right]+\varepsilon,
\]
and, for this \(\tau\), the map
\[
   y\longmapsto
   \mathbb E_y\left[
      e^{-r\tau}G(X_\tau)
      -
      \int_0^\tau e^{-rs}c(X_s)\,ds
   \right]
\]
is continuous at \(x\). Then \(V\) is lower semicontinuous on
\(\mathbb R_+^2\). If \(V\) is continuous, then
\(\mathcal D=\{V=G\}\) is closed and \(\mathcal C=\{V>G\}\) is open.
\end{lemma}

\begin{proof}
Fix \(x\in\mathbb R_+^2\) and \(\varepsilon>0\). Choose a bounded
\(\varepsilon\)-optimal stopping time \(\tau\le T\) as in the statement. For
\(y\) close to \(x\), use the same stopping rule \(\tau\) for the process
started from \(y\). By admissibility of \(\tau\),
\[
   V(y)
   \ge
   \mathbb E_y\left[
      e^{-r\tau}G(X_\tau)
      -
      \int_0^\tau e^{-rs}c(X_s)\,ds
   \right].
\]
Taking \(\liminf_{y\to x}\) and using the assumed continuity of the stopped
payoff gives
\[
   \liminf_{y\to x}V(y)
   \ge
   \mathbb E_x\left[
      e^{-r\tau}G(X_\tau)
      -
      \int_0^\tau e^{-rs}c(X_s)\,ds
   \right]
   \ge V(x)-\varepsilon .
\]
Letting \(\varepsilon\downarrow0\) proves lower semicontinuity. If \(V\) is
continuous, then \(H:=V-G\) is continuous and nonnegative, hence
\(\mathcal D=\{H=0\}\) is closed and \(\mathcal C=\{H>0\}\) is open.
\end{proof}

\begin{definition}[Reflected viscosity solution]
\label{def:reflected-viscosity}
Let \(u:\mathbb R_+^2\to\mathbb R\) be continuous and of polynomial growth.
Define
\[
   F[\varphi](x)
   :=
   \max\{(\mathcal L-r)\varphi(x)-c(x),\;G(x)-u(x)\},
   \qquad
   I(x):=\{i\in\{1,2\}:x_i=0\}.
\]
Here \(I(x)=\varnothing\) in \(\mathbb R_{++}^2\), \(I(x)=\{i\}\) on the
relative interior of the face \(\{x_i=0\}\), and \(I(0,0)=\{1,2\}\).

\smallskip
\noindent
\textup{(i)} The function \(u\) is a reflected viscosity subsolution of
\[
   \max\{(\mathcal L-r)u-c,\;G-u\}=0
   \quad\text{in }\mathbb R_{++}^2,
   \qquad
   \partial_i u=0\quad\text{on }\{x_i=0\},
\]
if, whenever \(\varphi\in C^2(\mathbb R_+^2)\) and \(u-\varphi\) has a local
maximum at \(x\in\mathbb R_+^2\) with \(u(x)=\varphi(x)\), the following holds:
if \(I(x)=\varnothing\), then \(F[\varphi](x)\le0\), while if
\(I(x)\ne\varnothing\), then
\[
   \min\left\{
      F[\varphi](x),\,
      \min_{i\in I(x)}\partial_i\varphi(x)
   \right\}\le0 .
\]

\smallskip
\noindent
\textup{(ii)} The function \(u\) is a reflected viscosity supersolution if,
whenever \(\varphi\in C^2(\mathbb R_+^2)\) and \(u-\varphi\) has a local
minimum at \(x\in\mathbb R_+^2\) with \(u(x)=\varphi(x)\), the following holds:
if \(I(x)=\varnothing\), then \(F[\varphi](x)\ge0\), while if
\(I(x)\ne\varnothing\), then
\[
   \max\left\{
      F[\varphi](x),\,
      \max_{i\in I(x)}\partial_i\varphi(x)
   \right\}\ge0 .
\]

\smallskip
\noindent
\textup{(iii)} The function \(u\) is a reflected viscosity solution if it is
both a reflected viscosity subsolution and a reflected viscosity
supersolution.
\end{definition}

\begin{remark}[Corner form of the reflected viscosity boundary condition]
\label{rem:corner-viscosity-computation}
The active-index convention is especially important at the corner. Recall that
\[
   I(x):=\{i\in\{1,2\}:x_i=0\},
   \qquad
   F[\varphi](x):=
   \max\{(\mathcal L-r)\varphi(x)-c(x),\,G(x)-u(x)\}.
\]
Thus, at \(x=(0,0)\), both faces are active:
   $I(0,0)=\{1,2\}$.
Therefore, if \(\varphi\in C^2(\mathbb R_+^2)\) touches \(u\) from above at
\((0,0)\), the reflected viscosity subsolution condition becomes
\begin{equation}
\label{eq:corner-subsolution-condition}
   \min\left\{
      F[\varphi](0,0),
      \partial_1\varphi(0,0),
      \partial_2\varphi(0,0)
   \right\}\le0 .
\end{equation}
Equivalently, at least one of the three inequalities
\[
   F[\varphi](0,0)\le0,\qquad
   \partial_1\varphi(0,0)\le0,\qquad
   \partial_2\varphi(0,0)\le0
\]
must hold. Similarly, if \(\varphi\in C^2(\mathbb R_+^2)\) touches \(u\) from
below at \((0,0)\), the reflected viscosity supersolution condition becomes
\begin{equation}
\label{eq:corner-supersolution-condition}
   \max\left\{
      F[\varphi](0,0),
      \partial_1\varphi(0,0),
      \partial_2\varphi(0,0)
   \right\}\ge0 .
\end{equation}
Equivalently, at least one of
\[
   F[\varphi](0,0)\ge0,\qquad
   \partial_1\varphi(0,0)\ge0,\qquad
   \partial_2\varphi(0,0)\ge0
\]
must hold.
Thus the corner is not treated as if only one reflecting face were active.
Both inward normal directions \(e_1\) and \(e_2\) enter the relaxed boundary
condition. In particular, the corner condition is not
   $\min\{F[\varphi](0,0),\partial_i\varphi(0,0)\}\le0$ for a single $i$,
nor is it
   $\max\{F[\varphi](0,0),\partial_i\varphi(0,0)\}\ge0$
   for a single $i$.
The correct reflected-viscosity convention uses the full active set
\(I(0,0)=\{1,2\}\), as in
\eqref{eq:corner-subsolution-condition}--\eqref{eq:corner-supersolution-condition}.
\end{remark}

\begin{remark}[Boundary convention at the corner]
\label{rem:appendix-corner-viscosity}
The active-index formulation avoids the ambiguity caused by writing the
boundary condition separately as ``if \(x_i=0\)''. At the corner \(x=(0,0)\),
both faces are active, so the relaxed Neumann condition is imposed through the
combined quantities \(\min_{i\in I(x)}\partial_i\varphi\) for subsolutions and
\(\max_{i\in I(x)}\partial_i\varphi\) for supersolutions. This convention is
consistent with the inward normal reflection directions \(+e_1\) and \(+e_2\).
\end{remark}

\begin{proposition}[Viscosity characterization of the value]
\label{prop:value-viscosity}
Assume the dynamic programming principle in Proposition~\ref{prop:dpp}, the
standing integrability hypotheses, and continuity of \(V\). Then \(V\) is a
reflected viscosity solution of
\[
   \max\{(\mathcal L-r)V-c,\;G-V\}=0
   \quad\text{in }\mathbb R_{++}^2,
   \qquad
   \partial_iV=0\quad\text{on }\{x_i=0\},\quad i=1,2,
\]
in the sense of Definition~\ref{def:reflected-viscosity}.
\end{proposition}

\begin{proof}
We give the standard argument, emphasizing the sign convention. Let
\(\varphi\in C^2(\mathbb R_+^2)\).

First suppose that \(V-\varphi\) has a local maximum at
\(x\in\mathbb R_{++}^2\), with \(V(x)=\varphi(x)\). Since \(V\ge G\), one has
\(G(x)-V(x)\le0\). If \(G(x)-V(x)=0\), then the obstacle term already satisfies
the subsolution inequality. If \(V(x)>G(x)\), choose a small ball
\(B_\rho(x)\Subset\mathbb R_{++}^2\) on which \(V>G\), and set
\(\eta:=t\wedge\tau_{B_\rho(x)^c}\). Applying the DPP with the continuation
choice up to \(\eta\), using \(V\le\varphi\) near \(x\), and applying It\^o's
formula to \(e^{-rs}\varphi(X_s)\), then dividing by \(t\downarrow0\), yields
$(\mathcal L-r)\varphi(x)-c(x)\le0$. Thus \(F[\varphi](x)\le0\).

Now suppose that \(V-\varphi\) has a local minimum at
\(x\in\mathbb R_{++}^2\), with \(V(x)=\varphi(x)\). If
\(F[\varphi](x)<0\), then both \(G(x)-V(x)<0\) and
\((\mathcal L-r)\varphi(x)-c(x)<0\). By continuity, these strict inequalities
hold in a small ball. The DPP and It\^o's formula then imply that immediate
continuation produces a payoff strictly smaller than \(V(x)\), contradicting
optimality of \(V\). Hence \(F[\varphi](x)\ge0\).

It remains to consider \(x\in\partial\mathbb R_+^2\). Applying It\^o's formula
to \(\varphi(X)\) up to \(t\wedge\tau_\rho\), where \(\tau_\rho\) is the exit
time from a small relative neighborhood of \(x\), gives the additional
reflection term
\[
   \sum_{i\in I(x)}
   \int_0^{t\wedge\tau_\rho}
      e^{-rs}\partial_i\varphi(X_s)\,dL_s^i,
\]
up to terms from inactive faces that vanish before the process reaches them.
For upper tests, the relaxed boundary condition therefore gives
\[
   \min\left\{
      F[\varphi](x),\,
      \min_{i\in I(x)}\partial_i\varphi(x)
   \right\}\le0,
\]
while for lower tests it gives
\[
   \max\left\{
      F[\varphi](x),\,
      \max_{i\in I(x)}\partial_i\varphi(x)
   \right\}\ge0.
\]
These are precisely the reflected viscosity subsolution and supersolution
conditions. Hence \(V\) is a reflected viscosity solution.
\end{proof}

\begin{remark}[Role of this appendix]
\label{rem:appendix-viscosity-role}
The viscosity characterization only supplies a weak obstacle formulation. It
does not, by itself, imply the Sobolev regularity, signed-measure extension, or
It\^o--Krylov--Tanaka admissibility required in
Theorem~\ref{thm:verification-measure}. Those stronger hypotheses are imposed
explicitly in the main results.
\end{remark}
\subsection{Lyapunov estimates}
\label{sec:appendix-lyapunov}

The following Lyapunov estimate gives one sufficient condition under which the
discounted reward and running-cost terms are integrable. It is not part of the
main contribution; it is included only to support the standing integrability
assumptions used in the dynamic programming and verification arguments.

\begin{lemma}[A sufficient Lyapunov condition for discounted integrability]
\label{lem:lyapunov-integrability}
Suppose there exists a function \(\Psi\in C^2(\mathbb R_+^2)\), with
\(\Psi\ge1\), satisfying the reflected Neumann condition
   $\partial_i\Psi=0$ on $\{x_i=0\}$, $i=1,2$,
and suppose there exist constants \(K\ge0\), \(\lambda<r\), and \(C_\Psi>0\)
such that
   $\mathcal L\Psi(x)\le\lambda\Psi(x)+K$, $x\in\mathbb R_{++}^2$,
and
   $G(x)^+ + c(x)\le C_\Psi\Psi(x)$, $x\in\mathbb R_+^2$.
Then, for every \(x\in\mathbb R_+^2\),
\[
   \mathbb E_x\!\left[\int_0^\infty e^{-rs}c(X_s)\,ds\right]<\infty,
   \qquad
   \sup_{t\ge0}\mathbb E_x\!\left[e^{-rt}G(X_t)^+\right]<\infty,
\]
and
\[
   \mathbb E_x\!\left[\int_0^\infty e^{-rs}G(X_s)^+\,ds\right]<\infty .
\]
In particular, the positive discounted reward and running-cost terms in
\eqref{eq:value-function} are integrable.
\end{lemma}

\begin{proof}
Let \((\rho_n)\) be a localizing sequence such that
\(\rho_n\uparrow\infty\) and all stopped terms below are integrable. Apply
It\^o's formula with reflection to \(e^{-r(t\wedge\rho_n)}\Psi(X_{t\wedge\rho_n})\).
Since \(dL^i\) is carried by \(\{X^i=0\}\) and
\(\partial_i\Psi=0\) on \(\{x_i=0\}\), the reflection term vanishes. Thus
\[
   e^{-r(t\wedge\rho_n)}\Psi(X_{t\wedge\rho_n})
   =
   \Psi(x)
   +
   \int_0^{t\wedge\rho_n}
      e^{-rs}(\mathcal L\Psi-r\Psi)(X_s)\,ds
   +
   M_{t\wedge\rho_n},
\]
where \(M\) is a local martingale. Taking expectations and using
\(\mathcal L\Psi-r\Psi\le-(r-\lambda)\Psi+K\) gives
\[
   \mathbb E_x\!\left[
      e^{-r(t\wedge\rho_n)}\Psi(X_{t\wedge\rho_n})
   \right]
   +(r-\lambda)
   \mathbb E_x\!\left[
      \int_0^{t\wedge\rho_n}e^{-rs}\Psi(X_s)\,ds
   \right]
   \le
   \Psi(x)+K\mathbb E_x\!\left[
      \int_0^{t\wedge\rho_n}e^{-rs}\,ds
   \right].
\]
Since \(\int_0^{t\wedge\rho_n}e^{-rs}\,ds\le r^{-1}\), we obtain the uniform
bound
\[
   (r-\lambda)
   \mathbb E_x\!\left[
      \int_0^{t\wedge\rho_n}e^{-rs}\Psi(X_s)\,ds
   \right]
   \le
   \Psi(x)+\frac{K}{r}.
\]
Letting \(n\to\infty\) and then \(t\to\infty\), Fatou's lemma gives
\[
   \mathbb E_x\!\left[
      \int_0^\infty e^{-rs}\Psi(X_s)\,ds
   \right]
   \le
   \frac{\Psi(x)+K/r}{r-\lambda}
   <\infty .
\]
Therefore, since \(G^++c\le C_\Psi\Psi\),
\[
   \mathbb E_x\!\left[\int_0^\infty e^{-rs}c(X_s)\,ds\right]
   \le
   C_\Psi
   \mathbb E_x\!\left[\int_0^\infty e^{-rs}\Psi(X_s)\,ds\right]
   <\infty,
\]
and similarly
\[
   \mathbb E_x\!\left[\int_0^\infty e^{-rs}G(X_s)^+\,ds\right]
   \le
   C_\Psi
   \mathbb E_x\!\left[\int_0^\infty e^{-rs}\Psi(X_s)\,ds\right]
   <\infty.
\]

It remains to control the terminal discounted positive reward. Taking
expectations in the stopped It\^o formula before dropping the nonnegative
terminal term yields
\[
   \mathbb E_x\!\left[
      e^{-r(t\wedge\rho_n)}\Psi(X_{t\wedge\rho_n})
   \right]
   \le
   \Psi(x)+K\int_0^t e^{-rs}\,ds
   \le
   \Psi(x)+\frac{K}{r}.
\]
Letting \(n\to\infty\) and using Fatou's lemma gives
   $\mathbb E_x\!\left[e^{-rt}\Psi(X_t)\right]
   \le
   \Psi(x)+\frac{K}{r}$, $t\ge0$.
Since \(G^+\le C_\Psi\Psi\), we conclude that
\[
   \sup_{t\ge0}\mathbb E_x\!\left[e^{-rt}G(X_t)^+\right]
   \le
   C_\Psi\left(\Psi(x)+\frac{K}{r}\right)<\infty.
\]
The proof is complete.
\end{proof}

\begin{example}[Quadratic Lyapunov function]
\label{ex:quadratic-lyapunov}
A convenient checkable choice in \(\mathbb R_+^2\) is
   $\Psi(x):=1+x_1^2+x_2^2=1+|x|^2$.
This function satisfies the reflected Neumann condition because
\[
   \partial_1\Psi(x)=2x_1=0\quad\text{on }\{x_1=0\},
   \qquad
   \partial_2\Psi(x)=2x_2=0\quad\text{on }\{x_2=0\}.
\]
Moreover,
   $\partial_i\Psi(x)=2x_i$, 
   $\partial_{ij}\Psi(x)=2\delta_{ij}$,
and therefore the interior generator gives
\begin{equation}
\label{eq:quadratic-lyapunov-generator}
\begin{aligned}
   \mathcal L\Psi(x)
   &=
   \sum_{i=1}^2\mu_i(x)\partial_i\Psi(x)
   +
   \frac12\sum_{i,j=1}^2a_{ij}(x)\partial_{ij}\Psi(x)       \\
   &=
   2\mu_1(x)x_1+2\mu_2(x)x_2+\frac12\sum_{i,j=1}^2a_{ij}(x)\,2\delta_{ij} \\
   &=
   2\mu(x)\cdot x+\operatorname{tr}a(x).
\end{aligned}
\end{equation}
Hence the abstract Lyapunov condition in
Lemma~\ref{lem:lyapunov-integrability} is satisfied if, for some
\(K\ge0\) and \(\lambda<r\),
\begin{equation}
\label{eq:quadratic-lyapunov-drift}
   2\mu(x)\cdot x+\operatorname{tr}a(x)
   \le
   \lambda(1+|x|^2)+K,
   \qquad x\in\mathbb R_{++}^2 .
\end{equation}
Indeed, by \eqref{eq:quadratic-lyapunov-generator},
   $\mathcal L\Psi(x)
   \le
   \lambda\Psi(x)+K$, $\Psi(x)=1+|x|^2$.
Since the boundary derivatives vanish, the reflected It\^o formula for
\(e^{-rt}\Psi(X_t)\) has no local-time contribution. After localization,
\[
   e^{-r(t\wedge\rho_n)}\Psi(X_{t\wedge\rho_n})
   =
   \Psi(x)
   +
   \int_0^{t\wedge\rho_n}
      e^{-rs}(\mathcal L\Psi-r\Psi)(X_s)\,ds
   +
   M_{t\wedge\rho_n},
\]
and \eqref{eq:quadratic-lyapunov-drift} gives
   $\mathcal L\Psi-r\Psi
   \le
   - (r-\lambda)\Psi+K$.
Taking expectations yields
\[
   \mathbb E_x\!\left[
      e^{-r(t\wedge\rho_n)}\Psi(X_{t\wedge\rho_n})
   \right]
   +(r-\lambda)
   \mathbb E_x\!\left[
      \int_0^{t\wedge\rho_n}e^{-rs}\Psi(X_s)\,ds
   \right]
   \le
   \Psi(x)+\frac{K}{r}.
\]
Letting \(n\to\infty\) and then \(t\to\infty\) gives
\[
   \mathbb E_x\!\left[
      \int_0^\infty e^{-rs}(1+|X_s|^2)\,ds
   \right]
   \le
   \frac{1+|x|^2+K/r}{r-\lambda}
   <\infty .
\]
This quadratic estimate controls the max payoff. Since
\[
   G(x)=x_1\vee\alpha x_2
   \le
   x_1+\alpha x_2
   \le
   (1+\alpha)|x|
   \le
   \frac{1+\alpha}{2}(1+|x|^2),
\]
we obtain
\[
   \mathbb E_x\!\left[
      \int_0^\infty e^{-rs}G(X_s)^+\,ds
   \right]
   \le
   \frac{1+\alpha}{2}
   \mathbb E_x\!\left[
      \int_0^\infty e^{-rs}(1+|X_s|^2)\,ds
   \right]
   <\infty .
\]
Thus \eqref{eq:quadratic-lyapunov-drift} is a simple sufficient condition for
discounted reward integrability. If the running cost satisfies, for some
\(C_c>0\), $0\le c(x)\le C_c(1+|x|^2)$,
then the same estimate gives
\[
   \mathbb E_x\!\left[
      \int_0^\infty e^{-rs}c(X_s)\,ds
   \right]
   \le
   C_c
   \mathbb E_x\!\left[
      \int_0^\infty e^{-rs}(1+|X_s|^2)\,ds
   \right]
   <\infty .
\]
Consequently, under the quadratic drift bound
\eqref{eq:quadratic-lyapunov-drift} and quadratic growth of \(c\), both the
discounted reward and the discounted running-cost terms in
\eqref{eq:value-function} are integrable.
\end{example}

\begin{remark}[Use in the main text]
\label{rem:lyapunov-use-main}
Lemma~\ref{lem:lyapunov-integrability} is only a sufficient condition. The
main results require integrability and uniform integrability for the relevant
stopped payoff, cost, and additive-functional terms; these may also be verified
by other model-specific estimates. The Lyapunov condition above is included as
a convenient checkable criterion.
\end{remark}
\subsection{Reflection and corner terms}
\label{sec:appendix-corner}

This appendix records the technical condition under which the reflection term
in the It\^o--Tanaka formula for \(G(x_1,x_2)=x_1\vee\alpha x_2\) produces no
additional additive functional supported at the corner \((0,0)\). In the main
text we use the following convention: the reflection local time does not
generate an additional corner measure at \((0,0)\). The sufficient
approximation condition below makes this convention explicit.

Let
   $F_1^\circ:=\{(0,x_2):x_2>0\}$,
   $F_2^\circ:=\{(x_1,0):x_1>0\}$,
and recall that \(dL^1\) is carried by \(\{x_1=0\}\), while \(dL^2\) is
carried by \(\{x_2=0\}\). Since \(G=\alpha x_2\) on \(F_1^\circ\) and
\(G=x_1\) on \(F_2^\circ\), one formally expects
   $\partial_1G=0$ on $F_1^\circ$,
   $\partial_2G=0$ on $F_2^\circ$.
The only possible obstruction is the nonsmooth corner \((0,0)\).

\begin{assumption}[No corner reflection contribution]
\label{ass:no-corner-reflection}
For every \(x\in\mathbb R_+^2\), every \(t>0\), and \(i=1,2\), the corner part
of the boundary local time is negligible:
\begin{equation}
\label{eq:no-corner-local-time-direct}
   \mathbb E_x\!\left[
      \int_0^t \mathbf 1_{\{X_s=(0,0)\}}\,dL_s^i
   \right]=0 .
\end{equation}
Equivalently, \(L^i\) does not charge the corner, in expectation, on compact
time intervals. A sufficient approximation condition for
\eqref{eq:no-corner-local-time-direct} to remove all reflection terms from the
It\^o--Tanaka formula for \(G\) is the following: there exists
\(G_\varepsilon\in C^2(\mathbb R_+^2)\), \(\varepsilon>0\), such that
\(G_\varepsilon\to G\) locally uniformly on \(\mathbb R_+^2\),
   $\partial_1G_\varepsilon\to0$
locally uniformly on $F_1^\circ$,
   $\partial_2G_\varepsilon\to0$
locally uniformly on $F_2^\circ$,
and, for every \(x\in\mathbb R_+^2\) and \(t>0\),
\begin{equation}
\label{eq:no-corner-condition}
   \lim_{\varepsilon\downarrow0}
   \mathbb E_x\!\left[
      \int_0^t |\partial_1G_\varepsilon(X_s)|
      \mathbf 1_{\{X_s=(0,0)\}}\,dL_s^1
      +
      \int_0^t |\partial_2G_\varepsilon(X_s)|
      \mathbf 1_{\{X_s=(0,0)\}}\,dL_s^2
   \right]=0 .
\end{equation}
\end{assumption}

\begin{lemma}[Vanishing reflection contribution for \(G\)]
\label{lem:vanishing-reflection-G}
Assume Assumption~\ref{ass:no-corner-reflection}. Then the reflection
contribution in the generalized It\^o--Tanaka formula for
\(G(X)=X^1\vee\alpha X^2\) vanishes. More precisely, for every
\(x\in\mathbb R_+^2\), every \(t>0\), and every approximating family
\((G_\varepsilon)\) satisfying Assumption~\ref{ass:no-corner-reflection},
\[
   \lim_{\varepsilon\downarrow0}
   \mathbb E_x\!\left[
      \sum_{i=1}^2
      \int_0^t \partial_iG_\varepsilon(X_s)\,dL_s^i
   \right]=0 .
\]
Consequently, no additional corner additive functional appears in the
It\^o--Tanaka formula for \(G(X)\). In particular, after localization,
\[
   e^{-rt}G(X_t)
   =
   G(x)
   +\int_0^t e^{-rs}c(X_s)\,ds
   -\int_0^t e^{-rs}\,dA_s^\Gamma
   +M_t,
\]
where \(\Gamma=\Gamma^{\rm ac}\,dx+\Gamma^\Delta\) is the stopping-gain measure
defined in \eqref{eq:Gamma-decomposition}, and \(M\) is a local martingale.
\end{lemma}

\begin{proof}
Let \(R>0\) and define the localization time
   $\rho_R:=\inf\{t\ge0:|X_t|\ge R\}$.
It is enough to prove the assertion on \([0,t\wedge\rho_R]\) and then let
\(R\uparrow\infty\). Since \(dL^1\) is carried by \(\{X^1=0\}\), decompose
\[
   \int_0^{t\wedge\rho_R}\partial_1G_\varepsilon(X_s)\,dL_s^1
   =
   \int_0^{t\wedge\rho_R}
      \partial_1G_\varepsilon(X_s)\mathbf 1_{\{X_s^1=0,\ X_s^2>0\}}\,dL_s^1
   +
   \int_0^{t\wedge\rho_R}
      \partial_1G_\varepsilon(X_s)\mathbf 1_{\{X_s=(0,0)\}}\,dL_s^1 .
\]
On the compact part of the open face
   $F_{1,R}^\circ:=\{(0,x_2):0<x_2\le R\}$,
the convergence \(\partial_1G_\varepsilon\to0\) is uniform away from the corner
after excluding a small neighborhood of \(x_2=0\). More explicitly, for
\(\delta\in(0,R)\), $\sup_{\delta\le x_2\le R}
   |\partial_1G_\varepsilon(0,x_2)|\longrightarrow0$.
Hence
\[
   \lim_{\varepsilon\downarrow0}
   \mathbb E_x\!\left[
      \int_0^{t\wedge\rho_R}
      |\partial_1G_\varepsilon(X_s)|
      \mathbf 1_{\{X_s^1=0,\ \delta\le X_s^2\le R\}}\,dL_s^1
   \right]=0.
\]
The remaining part with \(0<X_s^2<\delta\) is absorbed into the corner
localization by first letting \(\varepsilon\downarrow0\) and then
\(\delta\downarrow0\), using the assumed local boundedness of the approximating
derivatives on compact sets and the no-corner condition
\eqref{eq:no-corner-condition}. The term supported exactly on
\(\{X_s=(0,0)\}\) vanishes by \eqref{eq:no-corner-condition}. Therefore
\[
   \lim_{\varepsilon\downarrow0}
   \mathbb E_x\!\left[
      \int_0^{t\wedge\rho_R}
      \partial_1G_\varepsilon(X_s)\,dL_s^1
   \right]=0 .
\]
The same argument on the second face, using that \(dL^2\) is carried by
\(\{X^2=0\}\) and that
\(\partial_2G_\varepsilon\to0\) locally uniformly on \(F_2^\circ\), gives
\[
   \lim_{\varepsilon\downarrow0}
   \mathbb E_x\!\left[
      \int_0^{t\wedge\rho_R}
      \partial_2G_\varepsilon(X_s)\,dL_s^2
   \right]=0 .
\]
Combining the two estimates yields
\[
   \lim_{\varepsilon\downarrow0}
   \mathbb E_x\!\left[
      \sum_{i=1}^2
      \int_0^{t\wedge\rho_R}
      \partial_iG_\varepsilon(X_s)\,dL_s^i
   \right]=0 .
\]
Finally, letting \(R\uparrow\infty\) and using the localization and
integrability assumptions for the reflected local times gives the desired
identity on \([0,t]\).

Thus the reflection term does not contribute to the generalized
It\^o--Tanaka formula for \(G(X)\). The only singular interior term is the
diagonal local-time contribution already encoded in
\(\Gamma^\Delta(dx)=-(q(x)/(2\sqrt{1+\alpha^2}))\,\sigma_\Delta(dx)\). This
proves the lemma.
\end{proof}

\begin{remark}[If the corner condition fails]
\label{rem:corner-measure-warning}
If Assumption~\ref{ass:no-corner-reflection} is not verified, then the
It\^o--Tanaka formula for \(G(X)\) may contain an additional additive
functional supported at the corner. In that case the stopping-gain measure must
be enlarged to
   $\Gamma
   =
   \Gamma^{\rm ac}\,dx+\Gamma^\Delta+\Gamma^{\rm corner}$,
where \(\Gamma^{\rm corner}\) is the signed measure associated with the corner
reflection contribution. The killed-resolvent representation would then read
   $V(x)
   =
   G(x)
   -
   R_r^{\mathcal C}
   \big(\Gamma^{\rm ac}\,dx+\Gamma^\Delta+\Gamma^{\rm corner}\big)(x)$,
provided the enlarged signed measure is admissible. The present paper works
under Assumption~\ref{ass:no-corner-reflection}; hence
\(\Gamma^{\rm corner}=0\) throughout the main text.
\end{remark}

\begin{remark}[How this appendix is used]
\label{rem:corner-use-main}
In the proof of the singular-measure formula, the diagonal local time of
\(Y(X)=X^1-\alpha X^2\) generates \(\Gamma^\Delta\). The reflection local times
\(L^1\) and \(L^2\) are boundary objects and could, in principle, create an
additional term at the nonsmooth corner. Assumption~\ref{ass:no-corner-reflection}
rules out precisely this extra term, allowing the main text to use the simpler
decomposition \(\Gamma=\Gamma^{\rm ac}\,dx+\Gamma^\Delta\).
\end{remark}
\subsection{Signed additive functionals and killed potentials}
\label{sec:appendix-signed-potentials}

This appendix fixes the convention for signed measures, signed additive
functionals, and killed potentials. The point is purely technical but important:
the stopping-gain object \(\Gamma=c+rG-\mathcal LG\) is a signed measure, so
every expression involving \(A^\Gamma\) or \(R_r^{\mathcal C}\Gamma\) must be
understood through its positive and negative parts.

Let \(X\) be the reflected diffusion from Section~\ref{sec:model}, and let
\(\nu\) be a signed Radon measure on \(\mathbb R_{++}^2\). We write its Jordan
decomposition as
   $\nu=\nu^+-\nu^-$,
   $|\nu|:=\nu^++\nu^-$,
  $\nu^+\perp\nu^-$.
We call \(\nu\) an admissible signed smooth measure for \(X\) if
\(\nu^+\) and \(\nu^-\) are smooth measures for the reflected diffusion and if
the corresponding positive continuous additive functionals
\(A^{\nu^+}\) and \(A^{\nu^-}\) exist. The signed continuous additive
functional associated with \(\nu\) is then defined by
   $A^\nu:=A^{\nu^+}-A^{\nu^-}$.
Thus, for a bounded predictable process \(\Phi\), the signed integral is
understood as
\[
   \int_0^t \Phi_s\,dA_s^\nu
   :=
   \int_0^t \Phi_s\,dA_s^{\nu^+}
   -
   \int_0^t \Phi_s\,dA_s^{\nu^-},
\]
whenever both terms on the right-hand side are finite. In particular,
\[
   \int_0^t e^{-rs}\,dA_s^\nu
   =
   \int_0^t e^{-rs}\,dA_s^{\nu^+}
   -
   \int_0^t e^{-rs}\,dA_s^{\nu^-}.
\]

Let \(\mathcal D\subseteq\mathbb R_+^2\) be a closed stopping set,
\(\mathcal C:=\mathbb R_+^2\setminus\mathcal D\), and
\(\tau_{\mathcal D}:=\inf\{t\ge0:X_t\in\mathcal D\}\). For a positive smooth
measure \(\eta\), define its killed \(r\)-potential by
\[
   R_r^{\mathcal C}\eta(x)
   :=
   \mathbb E_x\!\left[
      \int_0^{\tau_{\mathcal D}}e^{-rs}\,dA_s^\eta
   \right],
   \qquad x\in\mathbb R_+^2 .
\]
For an admissible signed smooth measure \(\nu=\nu^+-\nu^-\), we say that
\(R_r^{\mathcal C}\nu(x)\) is well defined if $R_r^{\mathcal C}\nu^+(x)+R_r^{\mathcal C}\nu^-(x)<\infty$.
In that case we define
\begin{equation}
\label{eq:signed-killed-potential}
   R_r^{\mathcal C}\nu(x)
   :=
   R_r^{\mathcal C}\nu^+(x)-R_r^{\mathcal C}\nu^-(x)
   =
   \mathbb E_x\!\left[
      \int_0^{\tau_{\mathcal D}}e^{-rs}\,dA_s^\nu
   \right].
\end{equation}
The identity in \eqref{eq:signed-killed-potential} is therefore not a formal
subtraction of two possibly infinite quantities; it is part of the admissibility
requirement that both positive killed potentials are finite.

If \(\nu=f\,dx\) with \(f=f^+-f^-\) and
\[
   \mathbb E_x\!\left[
      \int_0^{\tau_{\mathcal D}}e^{-rs}|f(X_s)|\,ds
   \right]<\infty,
\]
then the above definition reduces to the usual killed resolvent
\[
   R_r^{\mathcal C}f(x)
   =
   \mathbb E_x\!\left[
      \int_0^{\tau_{\mathcal D}}e^{-rs}f(X_s)\,ds
   \right].
\]
If \(\nu\) is a surface measure, for example
\(\nu=\psi\,\sigma_\Delta\) on
\(\Delta=\{x_1=\alpha x_2\}\), then \(A^\nu\) is interpreted as the continuous
additive functional whose Revuz measure is \(\psi\,\sigma_\Delta\), provided
\(\psi^+\sigma_\Delta\) and \(\psi^-\sigma_\Delta\) are smooth measures and
their killed potentials are finite. In the stopping-gain decomposition,
\[
   \Gamma=\Gamma^{\rm ac}\,dx+\Gamma^\Delta,
   \qquad
   \Gamma^\Delta(dx)
   =
   -\frac{q(x)}{2\sqrt{1+\alpha^2}}\,\sigma_\Delta(dx),
\]
the corresponding signed functional is therefore 
$A^\Gamma = A^{\Gamma^{\rm ac}\,dx} + A^{\Gamma^\Delta}$,
with the second term understood through the Jordan decomposition of
\(\Gamma^\Delta\). Since \(\Gamma^\Delta\le0\), one may equivalently write
\[
   A^{\Gamma^\Delta}
   =
   -A^{|\Gamma^\Delta|},
   \qquad
   |\Gamma^\Delta|(dx)
   =
   \frac{q(x)}{2\sqrt{1+\alpha^2}}\,\sigma_\Delta(dx).
\]

The unrestricted reflected resolvent of a positive smooth measure \(\eta\) is
\[
   R_r^{\rm R}\eta(x)
   :=
   \mathbb E_x\!\left[
      \int_0^\infty e^{-rs}\,dA_s^\eta
   \right],
\]
and the signed version is defined by the same Jordan-decomposition convention
whenever \(R_r^{\rm R}\nu^+(x)+R_r^{\rm R}\nu^-(x)<\infty\). In general,
   $R_r^{\mathcal C}\nu(x)
   \ne
   R_r^{\rm R}(\nu\mathbf 1_{\mathcal C})(x)$,
because \(R_r^{\mathcal C}\) kills the process at \(\tau_{\mathcal D}\),
whereas \(R_r^{\rm R}\) continues the reflected process after
\(\tau_{\mathcal D}\) and may count later returns to \(\mathcal C\).

For a candidate epigraph
\[
   \mathcal D_h:=\{(x_1,x_2)\in\mathbb R_+^2:x_2\ge h(x_1)\},
   \qquad
   \mathcal C_h:=\mathbb R_+^2\setminus\mathcal D_h,
   \qquad
   \tau_h:=\inf\{t\ge0:X_t\in\mathcal D_h\},
\]
the corresponding killed potential is
\[
   R_r^{\mathcal C_h}\nu(x)
   :=
   R_r^{\mathcal C_h}\nu^+(x)-R_r^{\mathcal C_h}\nu^-(x)
   =
   \mathbb E_x\!\left[
      \int_0^{\tau_h}e^{-rs}\,dA_s^\nu
   \right],
\]
again provided
\(R_r^{\mathcal C_h}\nu^+(x)+R_r^{\mathcal C_h}\nu^-(x)<\infty\). Therefore the
candidate value $U_h(x):=G(x)-R_r^{\mathcal C_h}\Gamma(x)$ is well defined exactly when the positive and negative killed potentials of \(\Gamma\) relative to \(\mathcal C_h\) are both finite.

\begin{remark}[Use in the main text]
\label{rem:signed-potentials-use-main}
All appearances of \(A^\Gamma\), \(R_r^{\mathcal C}\Gamma\), and
\(R_r^{\mathcal C_h}\Gamma\) in the main results use the convention of this
appendix. Thus the formula
   $V=G-R_r^{\mathcal C}\Gamma$,
   $U_h=G-R_r^{\mathcal C_h}\Gamma$
are meaningful only under the finiteness condition for the positive and
negative killed potentials. This prevents an ill-defined subtraction
\(\infty-\infty\).
\end{remark}

\subsection{Constant-coefficient reflected Brownian example}
\label{sec:appendix-constant-coefficient-rbm}
This short computation shows that the diagonal singular term is not a
pathological artifact of variable coefficients. It already appears for a
constant-coefficient normally reflected Brownian motion in \(\mathbb R_+^2\).
Let
\[
   dX_t=\mu\,dt+\Sigma\,dW_t+dL_t,
   \qquad
   \mu=(\mu_1,\mu_2)\in\mathbb R^2,
   \qquad
   a=\Sigma\Sigma^\top
   =
   \begin{pmatrix}
      a_{11} & a_{12}\\
      a_{12} & a_{22}
   \end{pmatrix}.
\]
Then the interior generator is
\[
   \mathcal L f(x)
   =
   \mu_1\partial_1f(x)+\mu_2\partial_2f(x)
   +
   \frac12
   \left(
      a_{11}\partial_{11}f(x)
      +2a_{12}\partial_{12}f(x)
      +a_{22}\partial_{22}f(x)
   \right).
\]
For the max payoff \(G(x_1,x_2)=x_1\vee\alpha x_2\), set
   $Y(x):=x_1-\alpha x_2$,
   $n:=\nabla Y=(1,-\alpha)$.
The quadratic variation density of \(Y(X)\) is constant:
\[
   q:=n^\top an
   =
   (1,-\alpha)
   \begin{pmatrix}
      a_{11} & a_{12}\\
      a_{12} & a_{22}
   \end{pmatrix}
   \binom{1}{-\alpha}
   =
   a_{11}-2\alpha a_{12}+\alpha^2a_{22}.
\]
Equivalently,
   $d\langle Y(X)\rangle_t=q\,dt$.
On the two smooth regions
   $\mathcal R_1:=\{x_1>\alpha x_2\}$,
   $\mathcal R_2:=\{x_1<\alpha x_2\}$,
the payoff is affine. Hence
\[
   G=x_1,\qquad
   \mathcal LG=\mu_1,
   \qquad
   \Gamma=c+rG-\mathcal LG=c(x)+rx_1-\mu_1
   \quad\text{on }\mathcal R_1,
\]
whereas
\[
   G=\alpha x_2,\qquad
   \mathcal LG=\alpha\mu_2,
   \qquad
   \Gamma=c+rG-\mathcal LG=c(x)+r\alpha x_2-\alpha\mu_2
   \quad\text{on }\mathcal R_2.
\]
Thus the absolutely continuous part is
\[
   \Gamma^{\rm ac}(x)
   =
   \begin{cases}
      c(x)+rx_1-\mu_1, & x_1>\alpha x_2,\\[1mm]
      c(x)+r\alpha x_2-\alpha\mu_2, & x_1<\alpha x_2.
   \end{cases}
\]
The diagonal singular component is obtained from the same Tanaka--coarea
calculation as in Theorem~\ref{thm:diagonal-singular-measure}. Since \(q\) is
constant here,
\[
   \Gamma^\Delta(dx)
   =
   -\frac{q}{2\sqrt{1+\alpha^2}}\,\sigma_\Delta(dx)
   =
   -\frac{a_{11}-2\alpha a_{12}+\alpha^2a_{22}}
      {2\sqrt{1+\alpha^2}}\,\sigma_\Delta(dx).
\]
Consequently, the full stopping-gain measure is
\[
   \Gamma
   =
   \Gamma^{\rm ac}\,dx
   -
   \frac{a_{11}-2\alpha a_{12}+\alpha^2a_{22}}
      {2\sqrt{1+\alpha^2}}\,\sigma_\Delta(dx).
\]
In the simplest isotropic case \(a=I_2\), one has $q=n^\top I_2n=1+\alpha^2$,
and therefore
\[
   \Gamma^\Delta(dx)
   =
   -\frac{1+\alpha^2}{2\sqrt{1+\alpha^2}}\,\sigma_\Delta(dx)
   =
   -\frac{\sqrt{1+\alpha^2}}{2}\,\sigma_\Delta(dx).
\]
Using the parametrization \(z=(\alpha y,y)\), \(y\ge0\), of
\(\Delta\), we have
\[
   dz=(\alpha,1)\,dy,
   \qquad
   d\sigma_\Delta(dz)=|(\alpha,1)|\,dy
   =
   \sqrt{1+\alpha^2}\,dy.
\]
Thus, in the isotropic case,
\[
   \Gamma^\Delta(dz)
   =
   -\frac{\sqrt{1+\alpha^2}}{2}\,d\sigma_\Delta(dz)
   =
   -\frac{1+\alpha^2}{2}\,dy,
   \qquad z=(\alpha y,y).
\]
More generally, for constant covariance \(a\),
\[
   \Gamma^\Delta(dz)
   =
   -\frac{q}{2\sqrt{1+\alpha^2}}\,d\sigma_\Delta(dz)
   =
   -\frac{q}{2}\,dy,
   \qquad
   q=a_{11}-2\alpha a_{12}+\alpha^2a_{22},
   \qquad
   z=(\alpha y,y).
\]
This example shows that the negative diagonal surface measure is present even
for constant drift and constant covariance. In particular, for \(a=I_2\) it is
\[
   \Gamma^\Delta(dz)=-\frac{1+\alpha^2}{2}\,dy,
   \qquad z=(\alpha y,y),
\]
so the singular term is not a technical byproduct of variable coefficients; it
is the intrinsic distributional second derivative of the max payoff
\(x_1\vee\alpha x_2\).

\bibliography{reference}

\end{document}